\documentclass[11pt]{amsart}
\usepackage{amscd,amssymb,amsfonts,amsthm}
\usepackage[arrow,matrix]{xy}
\usepackage{xcolor}
\usepackage{mathrsfs}
\usepackage{mathtools}
\usepackage{enumitem}
\usepackage{graphicx}
\usepackage{epsfig}
\usepackage{lastpage}
\usepackage{titleps}
\usepackage[justification=centering]{caption}
\usepackage[colorlinks=true,linkcolor=blue,citecolor=purple]{hyperref}

\usepackage[backend=bibtex,style=numeric,natbib=true]{biblatex}
\usepackage[margin=2cm,top=2.5cm,headheight=16pt,headsep=0.25in,heightrounded]{geometry}
\newpagestyle{main}{
\setheadrule{.3pt} % Header rule
\sethead[\thepage][\textsc{tri-linear birational maps in dimension three}][] % even pages
{}{\textsc{laurent bus\'e},\
\textsc{pablo gonz\'alez-maz\'on},\
\textsc{josef schicho}
}{\thepage} % odd pages
} 
\usepackage{tikz-cd}
\tikzstyle{arrow}=[draw, -latex]
\theoremstyle{definition}
\newtheorem{lemma}{Lemma}[section]
\newtheorem{proposition}[lemma]{\textbf{Proposition}}
\newtheorem{theorem}[lemma]{\textbf{Theorem}}
\newtheorem{corollary}[lemma]{\textbf{Corollary}}

\theoremstyle{definition}

\newtheorem{example}[lemma]{\textbf{Example}}

\newtheorem{remark}[lemma]{Remark}
\newtheorem{condition}[lemma]{Condition}
\newtheorem{setup}[lemma]{Setup}
\numberwithin{equation}{section}
\newcommand{\C}{\mathbb{C}}

\newcommand{\Z}{\mathbb{Z}}

\renewcommand{\P}{\mathbb{P}}
\newcommand{\Oc}{\mathcal{O}}

\newcommand{\Cc}{\mathscr{C}}

\newcommand{\PGL}{\mathrm{PGL}}
\newcommand{\Aut}{\mathrm{Aut}}
\newcommand{\Bir}{\mathfrak{Bir}}
\newcommand{\Birr}{\mathrm{Bir}}

\newcommand{\Proj}{\text{Proj }}

\newcommand{\blowup}{\text{Bl}}
\newcommand{\Rat}{\mathfrak{Rat}}

\newcommand{\HS}{\mathrm{HS}}

\newcommand{\HP}{\mathrm{HP}}

\begin{document}

% Title %%%%%%%%%%%%%%%%%%%%%%%%%%%%%%%%%%%%%%%%%%%%%%%%%%%%%%%%%%%%%%%%%%%%%%%%%%%%%%%%%%%%%%
\title{Tri-linear birational maps in dimension three}

% Authors %%%%%%%%%%%%%%%%%%%%%%%%%%%%%%%%%%%%%%%%%%%%%%%%%%%%%%%%%%%%%%%%%%%%%%%%%%%%%%%%%%%%
\author{Laurent Busé}
\author{Pablo Gonz\'alez-Maz\'on}
\address{Université Côte d'Azur, Inria, 2004 route des Lucioles, 06902 Sophia Antipolis, France.}
\email{laurent.buse@inria.fr, pablo.gonzalez-mazon@inria.fr}
\author{Josef Schicho}
\address{Johannes Kepler University, RISC, Austria.}
\email{jschicho@risc.jku.at}

% Abstract %%%%%%%%%%%%%%%%%%%%%%%%%%%%%%%%%%%%%%%%%%%%%%%%%%%%%%%%%%%%%%%%%%%%%%%%%%%%%%%%%%%
\begin{abstract}
A tri-linear rational map in dimension three is a rational map $\phi: (\mathbb{P}_\mathbb{C}^1)^3 \dashrightarrow \mathbb{P}_\mathbb{C}^3$ defined by four tri-linear polynomials without a common factor. If $\phi$ admits an inverse rational map $\phi^{-1}$, it is a tri-linear birational map. In this paper, we address computational and geometric aspects about these transformations. 
We give a characterization of birationality based on the first syzygies of the entries. More generally, we describe all the possible minimal graded free resolutions of the ideal generated by these entries. With respect to geometry, we show that the set $\Bir_{(1,1,1)}$ of tri-linear birational maps, up to composition with an automorphism of $\P_\C^3$, is a locally closed algebraic subset of the Grassmannian of $4$-dimensional subspaces in the vector space of tri-linear polynomials, and has eight irreducible components. Additionally, the group action on $\Bir_{(1,1,1)}$ given by composition with automorphisms of $(\P_\C^1)^3$ defines 19 orbits, and each of these orbits determines an isomorphism class of the base loci of these transformations. 
\end{abstract}
%%%%%%%%%%%%%%%%%%%%%%%%%%%%%%%%%%%%%%%%%%%%%%%%%%%%%%%%%%%%%%%%%%%%%%%%%%%%%%%%%%%%%%%%%%%%%%

\maketitle

\vspace{-.3cm}

\section{Introduction}
\label{section: introduction}

\subsection*{Motivation} Rational maps between multi-projective and projective spaces have been studied recently, with a special emphasis on the computation of the degrees and birationality  (e.g.~\cite{BCD20,multiview,EFFECTIVE_BIGRAD,SGW16,floater}). The introduction of the multi-projective setting is motivated by applications of these rational maps to Computer-Aided Geometric Design (CAGD), computer vision, shape optimization, animation, and other related fields (\cite{SGW16} and references therein). In particular, maps defined over products of projective lines are ubiquitous in geometric modeling. For instance, tensor-product rational B\'ezier surfaces and volumes are used intensively for the definition of geometric objects. More specifically, 
dominant rational maps $(\P_\C^1)^2 \dashrightarrow \P_\C^2$ and $(\P_\C^1)^3 \dashrightarrow \P_\C^3$ provide simple and meaningful control points for the deformation of 2D and 3D shapes, in a technique known as Free-Form Deformation (FFD) (e.g.~\cite{sederberg2D,SGW16}). 
One of the requirements for these applications is the injectivity of the rational map. More interestingly, birational maps are globally injective, and have an inverse rational map. Moreover, the computation of preimages is necessary in some applications (\cite{SGW16} and references), and the inverse map allows the computation of these preimages avoiding the numerical resolution of equations.

\subsection*{Previous works and context} 
Approximately thirty years ago, modern techniques from commutative algebra were incorporated to the  study of birational maps, primarily motivated by computational purposes. Syzygy-based birationality criteria date back to \cite{SHK}, and have been further developed ever since (e.g.~\cite{RS,S}). In this context, the Jacobian Dual Criterion (JDC) is a general method to decide if a rational map between projective varieties, or from a multi-projective variety to a projective variety, is birational \cite{JDC,BCD20}. The JDC relies on the computation of the defining equations of the Rees algebra associated to the rational map. Unfortunately, the derivation of these equations entails a high computational cost, and for some applications specialized methods are preferred. 
In dimension two, there are efficient strategies for the manipulation of bi-linear birational maps $\phi: (\P_\C^1)^2 \dashrightarrow \P_\C^2$, whose entries are polynomials of bi-degree $(1,1)$ without a common factor, where a syzygy-based characterization of birationality is known \cite{sederberg2D,EFFECTIVE_BIGRAD}. Larger bi-degrees have also been considered, but only partially. Namely, sufficient conditions under which birationality holds are known for maps with entries of bi-degree $(2,2)$ \cite{EFFECTIVE_BIGRAD} and $(1,n)$ \cite{SGW16}.
In dimension three, no syzygy-based criteria for the birationality of a rational map $\phi: (\P_\C^1)^3 \dashrightarrow \P_\C^3$ are known yet, probably because the situation is much more involved.
\vskip0.2cm
On the other hand, the study of groups and spaces of birational maps has a long history in algebraic geometry, and there is an extensive literature about the topic (e.g.~\cite{hudson,Alberich,petitdegree,bisiplane,topologiesAndStructures}). The $n$-dimensional Cremona group over $\C$, denoted by either $\text{Cr}_n(\C)$ or $\Birr(\P_\C^n)$, is the group of birational endomorphisms of $\P_\C^n$.  Given $d\geq 0$, the subset $\mathfrak{Bir}_d(\P_\C^n) \subset \text{Cr}_n(\C)$ consists of all the birational endomorphisms that are determined by $n+1$ homogeneous polynomials of degree $d$, without a common factor. 
\vskip0.2cm
Remarkably, we have a good understanding of $\Bir_d(\P_\C^2)$  \cite{bisiplane} and a complete classification of the base loci is known for the degrees $d = 2,3$ \cite{petitdegree}. However, for $n\geq 3$ the situation quickly becomes more complicated. For $n=3$, the algebraic structure of $\Bir_2(\P_\C^3)$ has been described together with a complete classification of the base loci of birational endomorphisms with quadratic entries \cite{transformationsquadratiques}, but this is no longer the case if $d\geq 3$. For $d=3$, some of the irreducible components of $\Bir_3(\P_\C^3)$ have also been described \cite{cubic}, but only a partial classification of the base loci is known \cite{hudson,cubic}. 
With respect to birational maps from a multi-projective space to a projective target, only the structure and classification of the base loci of bi-linear birational maps $(\P_\C^1)^2 \dasharrow \P_\C^2$ have been addressed \cite{sederberg2D,EFFECTIVE_BIGRAD}. Namely,
the set of bi-linear birational maps is  an irreducible hypersurface in $\P_\C^{11}$, and the base locus is always a closed point.  
\vskip0.2cm
In this paper, we study the algebraic set of tri-linear birational maps $\phi: (\P_\C^{1})^3 \dasharrow \P_\C^3$ and we give the complete classification of the base loci of such transformations, up to isomorphism of schemes. This geometric analysis targets the derivation of  computational methods for the manipulation of these  birational maps. Namely, we provide a syzygy-based characterization for birationality, whose proof strongly relies on the classification of the base loci. 

\subsection*{Organization of the paper} In §\ref{section: preliminars} we introduce tri-linear birational maps in dimension three and  the concepts of \textit{base ideal}, \textit{base locus}, and \textit{type} of a tri-linear birational map. Additionally, we briefly review the Jacobian Dual Criterion. In §\ref{section: factorization of tri-linear birational maps} we explain our strategy for the factorization of tri-linear birational maps, which is crucial for the later analysis.   
The algebraic structure of the set $\mathfrak{Bir}_{(1,1,1)}$, consisting of tri-linear birational maps up to composition with automorphisms of $\P_\C^3$, is studied in §\ref{section: the algebraic set of tri-linear birational maps}. Here, we also describe  the group action on $\mathfrak{Bir}_{(1,1,1)}$ given by the automorphisms of $(\P_\C^1)^3$, that we call \textit{right-action}. 
Our first main result is Theorem \ref{theorem: algebraic structure of Bir(1,1,1)}, which states that $\mathfrak{Bir}_{(1,1,1)}$ has the structure of a locally closed algebraic subset of a certain Grassmannian and describes its irreducible components. 
Later, §\ref{section: classification} is devoted to the classification of the base loci. In Theorems \ref{theorem: classification (1,1,1)}, \ref{theorem: classification (1,1,2)}, \ref{theorem: classification (1,2,2)}, and \ref{theorem: classification (2,2,2)} we derive the complete list of the orbits that the right-action defines on $\Bir_{(1,1,1)}$, which are in correspondence with the isomorphism classes of the base loci. 
On the computational side, in §\ref{section: syzygy-based characterization} we study the syzygies of the entries of tri-linear birational maps. Namely, Theorem \ref{theorem: main computational result} is a characterization of birationality based on the first syzygies, whose proof relies  strongly on the classification of §\ref{section: classification}. More generally, in Propositions \ref{proposition: syzygies (1,1,1)}, \ref{proposition: syzygies (1,1,2)}, \ref{proposition: syzygies (1,2,2)}, and \ref{proposition: syzygies (2,2,2)} we also classify all the possible minimal tri-graded free resolutions of the base ideal of a tri-linear birational map.

\subsection*{Acknowledgements} The authors are grateful to the anonymous reviewers for their comments which significantly improved the content and presentation of this paper. The three authors have received funding from the European Union’s ``\textsc{Hori\-zon 2020 research and innovation programme}'', under the Marie Skłodowska-Curie grant agreement n$^\circ$ 860843. All the examples, and multiple computations not explicitly included, were performed using the computer algebra software \textsc{Macaulay2} \cite{M2}.

\section{Preliminars}
\label{section: preliminars}

\subsection{Tri-linear birational maps in dimension three}
\label{subsection: tri-linear rational maps}

A tri-linear rational map in dimension three $\phi$ is a rational map
\begin{eqnarray}
\label{eq: tri-linear rational map}
	\phi: \P_\C^1\times \P_\C^1 \times \P_\C^1 & \dasharrow & \P_\C^3 \\ \nonumber
	(x_0:x_1)\times(y_0:y_1)\times(z_0:z_1) & \mapsto & (f_0:f_1:f_2:f_3)\ ,
\end{eqnarray}
where $f_i = f_i(x_0,x_1,y_0,y_1,z_0,z_1)$ is a tri-linear polynomial, i.e$.$ is linear and homogeneous with respect to the three pairs of variables $\{x_0,x_1\}$, $\{y_0,y_1\}$, and $\{z_0,z_1\}$, and the $f_i$'s do not have a common factor. For simplicity in the notation, we will refer to them simply as tri-linear rational maps. If $\phi$ admits an inverse rational map, then $\phi$ is a tri-linear birational map.
\vskip0.2cm
Let $A_1=\C[x_0,x_1]$, $A_2=\C[y_0,y_1]$, and $A_3=\C[z_0,z_1]$ be the graded homogeneous coordinate rings of each $\P_\C^1$, with the standard grading in each ring given by the corresponding pair of variables. The tri-homogeneous coordinate ring of the product $X\coloneqq(\P_\C^1)^3$ is the tensor product ring $R\coloneqq A_1\otimes_\C A_2 \otimes_\C A_3$, which inherits a tri-graded structure $$R=\bigoplus_{(i,j,k)\in\mathbb{Z}^3} \, R_{(i,j,k)}
=
\bigoplus_{(i,j,k)\in\mathbb{Z}^3} \, (A_1)_i \otimes_\C (A_2)_j \otimes_\C (A_3)_k
\ ,$$
where subindices stand for graded components. The irrelevant ideal of $R$ is
$$
\mathfrak{N} = 
\bigoplus_{i>0,\,j>0,\,k>0} \, (A_1)_i \otimes_\C (A_2)_j \otimes_\C (A_3)_k
=
(x_0,x_1)
\cap
(y_0,y_1)
\cap
(z_0,z_1)
\ .
$$
The tri-homogeneous ideal $B_\phi =(f_0,f_1,f_2,f_3)\subset R$ is the \textit{base ideal} of $\phi$. The subscheme $Z_\phi$ of $X$ defined by the base ideal $B_\phi$ is the \textit{base locus} or \textit{base locus scheme} of $\phi$. The codimension of $Z_\phi$ is always at least two, as otherwise we find a common factor to the $f_i$'s. On the other hand, a surface $S\subset X$ is the vanishing locus of a tri-homogeneous polynomial in $R$. The \textit{tri-degree of }$S$, denoted by tri-deg$(S)$, is the tri-degree of its defining polynomial.

\subsection{Curves in $(\P_\C^1)^3$}
\label{subsection: curves in X} 
Let $Y$ be a one-dimensional subscheme of $X$. 
The \textit{tri-degree of} $Y$, denoted by $\text{tri-deg}(Y)$, is the triple $(d_1,d_2,d_3)$ where $d_i$ is the number of points, counted with multiplicity, in the intersection of $Y$ with the subspace of $X$ determined by a general linear form in $A_i$. Equivalently, $d_i$ is the degree of the projection of $Y$ onto the $i$-th $\P_\C^1$ factor of $X$. For zero-dimensional subschemes of $X$, we define the tri-degree as $(0,0,0)$.
\vskip0.2cm
If all the irreducible components of $Y$ have dimensione one, we say that $Y$ is a curve. In this situation, we write $I_Y\subset R$ for its defining ideal.
On the other hand, recall that $Y$ is Cohen-Macaulay if for every point $P\in Y$ the local ring of $Y$ at $P$ is Cohen-Macaulay. In particular, $Y$ is a Cohen-Macaulay curve if and only if it does not have embedded nor isolated points \cite[Lemma 31.4.4]{stacks-project}. 
\vskip0.2cm
We define $C_Y\subset X$ as the subscheme determined by the intersection of all the associated primes to $I_Y$ in $R$ of codimension exactly two. Clearly, $C_Y\subset Y$ is a Cohen-Macaulay curve. Additionally, when $Y = Z_\phi$ we denote $C_Y$ simply by $C_\phi$, and $I_{C_Y}$ simply by $I_\phi$.
\vskip0.2cm
\subsection{The Jacobian Dual Criterion} 
\label{subsection: jacobian dual criterion}
Let $J$ be the kernel of the ring homomorphism
\begin{eqnarray*}
	R\, [t_0,t_1,t_2,t_3] & \rightarrow & R\, [u] \\
	t_i & \mapsto & f_i\, u \ \ \ \ \ \ \ .
\end{eqnarray*} 
The ideal $J$ is the \textit{Rees ideal} associated to $\phi$. 
It is a prime ideal hence saturated with respect to $\mathfrak{N}$. Moreover, it defines the scheme-theoretic graph of $\phi$ in the product $X \times \P_\C^3$. 
\vskip0.2cm
The Rees ideal $J$ is equipped with a $\Z^4$-graded structure, as it inherits the $\Z^3$-grading from $R$ and the standard $
\Z$-grading from $\C[t_0,t_1,t_2,t_3]$. We use the notation $(d_1,d_2,d_3;d_4)$ for this multi-degree in $\mathbb{Z}^4$. More explicitly, we have
$$g = g(x_0,x_1,y_0,y_1,z_0,z_1,t_0,t_1,t_2,t_3) \in J_{(d_1,d_2,d_3;d_4)}$$
if $g$ is homogeneous of degree $d_1$ in  $x_0,x_1$, of degree $d_2$ in $y_0,y_1$, of degree $d_3$ in $z_0,z_1$, and of degree $d_4$ in $t_0,t_1,t_2,t_3$, and additionally 
$$
g(x_0,x_1,y_0,y_1,z_0,z_1,f_0,f_1,f_2,f_3)=0\ .
$$
The elements in $J_{(d_1,d_2,d_3;1)}$ are in correspondence with the first syzygies of the $f_i$'s with coefficients of tri-degree $(d_1,d_2,d_3)$ in $R$. We refer to these as syzygies of tri-degree $(d_1,d_2,d_3)$.
\vskip0.2cm
The multi-graded Jacobian Dual Criterion (JDC) \cite[Theorem 4.4]{BCD20} characterizes birationality in terms of the ideal $J$. Namely, given $(d_1,d_2,d_3)\in \Z^3$ consider the $\C[t_0,t_1,t_2,t_3]$-module 
$$
J_{(d_1,d_2,d_3;*)} \coloneqq \bigoplus_{d_4\in \, \Z} J_{(d_1,d_2,d_3;d_4)}\ .
$$
Assuming that $\phi$ is dominant, the JDC establishes that $\phi$ is birational if and only if the modules 
$$
J_{(1,0,0;*)} \ ,\ J_{(0,1,0;*)}\ ,\ J_{(0,0,1;*)}
$$ 
all have rank one when regarded as vector spaces over the field $\C(t_0,t_1,t_2,t_3)$. 
In particular, we find minimal integers $r_1,r_2,r_3\in \mathbb{Z}$ such that 
\begin{equation}
\label{eq: inverse and syzygies}
\begin{vmatrix}
x_0 & x_1 \\
a_0 & a_1
\end{vmatrix} 
\in J_{(1,0,0;\, r_1)} 
\ ,\
\begin{vmatrix}
y_0 & y_1 \\
b_0 & b_1
\end{vmatrix}
\in J_{(0,1,0;\, r_2)} 
\ ,\
\begin{vmatrix}
z_0 & z_1 \\
c_0 & c_1
\end{vmatrix} 
\in J_{(0,0,1;\, r_3)}\ ,
\end{equation}
for some polynomials $a_i,b_j,c_k\in \C[t_0,t_1,t_2,t_3]$. 
Thus, we find $s,t,u\in R$ satisfying 
$$
a_i 
\xmapsto{t_l\mapsto f_l}
x_i\, s
\ ,\ 
b_j 
\xmapsto{t_l\mapsto f_l}
y_j\, t
\ ,\
c_k
\xmapsto{t_l\mapsto f_l}
z_k \, u
\ ,
$$
and the rational map
\begin{eqnarray}
\label{eq: rational inverse}
	\phi^{-1}: \P_\C^3 & \dasharrow & X = \P_\C^1 \times \P_\C^1 \times \P_\C^1 \\ \nonumber
	(t_0:t_1:t_2:t_3) & \mapsto & (a_0:a_1)\times(b_0:b_1)\times(c_0:c_1) 
\end{eqnarray}
must be the inverse of $\phi$, as the composition $\phi^{-1}\circ\phi$ yields the identity in any open subset of $X$ where $s,t$, and $u$ are all non-zero.
\subsection{The type of $\phi$}
\label{subsection: possible types of inverse}
Assume that $\phi$ is birational, and  maintain the notation of §\ref{subsection: jacobian dual criterion}. We define the \textit{type of} $\phi$ to be the triple $
(r_1 , r_2 , r_3)
=
(\deg a_i , \deg b_j , \deg c_k)\in\Z^3$. Notice that this is well defined, as we are assuming $a_0,a_1$ (resp$.$ $b_0,b_1$ and $c_0,c_1$) of minimal degree.
\vskip0.2cm
Given a general line $\ell\subset \P_\C^3$ the image $\phi^{-1}(\ell)$ is a closed rational curve in $X$. Additionally, writing $\pi_i : X \xrightarrow{} \P_\C^1$ for the projection of $X$ onto the $i$-th factor, the composition $\pi_1 \circ \phi^{-1}|_{\ell}$
\begin{alignat*}{5}
    \P_\C^1 \cong \ell \subset \P_\C^3 & \xrightarrow{\phi^{-1}|_{\ell}} & \phi^{-1}(\ell) \subset X = \P_\C^1 \times \P_\C^1 \times \P_\C^1 & \xrightarrow{\pi_1}  && \ \ \ \P_\C^1 \\
    (t_0:t_1:t_2:t_3) & \mapsto & (a_0:a_1)\times (b_0:b_1)\times (c_0:c_1) & \mapsto && (a_0:a_1)\ ,
\end{alignat*}
yields an endomorphism of $\P_\C^1$ of degree $r_1$. Similarly, the projections onto the other factors yield endomorphisms of degree $r_2$ and $r_3$, implying that the tri-degree of $\phi^{-1}(\ell)$ is $(r_1,r_2,r_3)$. 
On the other hand, let $(\sum_{i=0}^3 \alpha_i t_i, \sum_{i=0}^3 \beta_i t_i )$ be the defining ideal of $\ell$ in $\C [t_0,t_1,t_2,t_3]$. Then, the tri-homogeneous ideal 
$$
\left( \alpha_0 f_0 +  \alpha_1 f_1 +  \alpha_2 f_2 +  \alpha_3 f_3 \, , \, \beta_0 f_0 +  \beta_1 f_1 +  \beta_2 f_2 +  \beta_3 f_3  \right)
$$
defines a scheme in $X$ of tri-degree $(2,2,2)$, which is the scheme-theoretic union of $\phi^{-1}(\ell)$ and the base locus $Z_\phi$. Therefore, it follows the identity
\begin{equation}
\label{eq: equality on type of inverse}
\text{tri-deg}(\phi^{-1}(\ell)) + \text{tri-deg}( Z_\phi ) = (r_1,r_2,r_3) + \text{tri-deg}(Z_\phi ) = (2,2,2)
\end{equation}
(a projective version of \eqref{eq: equality on type of inverse} can be found in \cite[§$1$]{cubic}). In particular, we obtain the following pairing between the possible types of $\phi$ and the tri-degree of the base locus $Z_\phi$.
\begin{table}[h]
    \centering
    \begin{tabular}{c|c}
    \hline
      Type of $\phi$   & Tri-degree of $Z_\phi$ \\
    \hline\hline
      (1,1,1)  &  (1,1,1) \\
      (1,1,2) (and permutations) &  (1,1,0) (and permutations) \\
      (1,2,2) (and permutations)  &  (1,0,0) (and permutations) \\
      (2,2,2)  &  (0,0,0) \\
    \hline
    \end{tabular}
    \newline
    \centering
    \caption{All the possible types of $\phi$, and the tri-degrees of $Z_\phi$.}
    \label{table: possible types of inverse}
\end{table}
\vskip-0.5cm
\noindent From \textsc{Table} \ref{table: possible types of inverse} it follows that if $\phi$ has type $(2,2,2)$ the base locus $Z_\phi$ is zero-dimensional.
\vskip0.4cm
\subsection{Automorphisms of $(\P_\C^1)^3$}
\label{subsection: automorphisms of X}
Let $\Aut(X)$ be the group of automorphisms of $X$. Consider the normal subgroup $H\trianglelefteq \Aut(X)$ of all the automorphisms that preserve the tri-degrees  in $X$, i.e$.$ those of the form
$$
(x_0:x_1)\times (y_0:y_1)\times (z_0:z_1)
\mapsto
\eta_1(x_0:x_1)\times \eta_2(y_0:y_1)\times \eta_3(z_0:z_1)
\ ,
$$
for some $\eta_i$'s in the projective general linear group $\PGL (2,\C)$. Clearly, we have an isomorphism $H\cong \PGL (2,\C)^3$. On the other hand, the automorphism given by
\begin{eqnarray*}
	X & \xrightarrow{} & X \\ \nonumber
	(x_0:x_1)\times(y_0:y_1)\times(z_0:z_1)  & \mapsto & (y_0:y_1)\times(z_0:z_1)\times(x_0:x_1) 
\end{eqnarray*}
sends curves of tri-degree $(d_1,d_2,d_3)$ to curves of tri-degree $(d_2,d_3,d_1)$, so it does not preserve the tri-degrees in $X$. In general, every  automorphism of $X$ sends curves of tri-degree $(d_1,d_2,d_3)$ to curves of tri-degree $(d_i,d_j,d_k)$ for some indices $i,j,k$ such that $\{i,j,k\}=\{1,2,3\}$, i.e$.$ it might permute the tri-degrees. 
Let $K\leq \Aut(X)$ be the subgroup consisting of the  automorphisms that simply permute the factors of $X$, which is isomorphic to the symmetric group $\mathfrak{S}_3$ acting on three elements. In particular, the intersection $H\cap K$ is trivial. Moreover, the association 
\begin{eqnarray}
	\label{eq: semi-direct product}
	K\times H & \xrightarrow{} & \Aut(X) \\
	 \nonumber
	(\sigma,\eta) & \mapsto & \sigma\circ\eta
\end{eqnarray}
is a bijection of sets. 
Therefore, the semidirect product of groups 
$
K \rtimes H \cong \mathfrak{S}_3 \rtimes \PGL(2,\C)^3
$ 
is isomorphic to $\Aut(X)$. 

\section{Factorization of tri-linear birational maps}
\label{section: factorization of tri-linear birational maps}

In this section we explain the strategy followed in our geometric analysis. It is inspired by \cite{transformationsquadratiques}, where the classification of birational endomorphisms of $\P_\C^3$ with quadratic entries is addressed. The idea is the factorization of tri-linear rational maps through a linear system and a  projection.

\subsection{Factorization of tri-linear rational maps}
\label{subsection: factorization of tri-linear rational maps}

Let $W$ be a $\C$-vector space satisfying 
\begin{equation}
\label{eq: vector space W}
(B_\phi)_{(1,1,1)} \subset W
\subset R_{(1,1,1)}
\ ,
\end{equation}
and let $N + 1 = \dim_\C W$. Given a basis of the dual vector space $W^\vee$, let $\zeta_{W}: X \dashrightarrow \P (W^\vee )$ be the rational map whose entries are the dual vectors in this basis, and  write $Y_W$ for the Zariski closure in $\P(W^\vee)$ of the image of $\zeta_{W}$. In particular, we have the commutative diagram

\begin{equation}
\label{eq: commutative diagram}
\begin{tikzcd}
& Y_W \arrow[r,hookrightarrow] &  \P \left( W^\vee \right) \cong \P_\C^N \arrow[d,dashrightarrow, "\pi_L"] \\
\blowup_W X \subset X\times \P_\C^{N} \arrow[r,rightarrow, "\pi_W"] \arrow[ur,rightarrow,"\Pi_W"] & X \arrow[r,dashrightarrow, "\phi"] \arrow[u,dashrightarrow, "\zeta_W"] & \P_\C^3 
\end{tikzcd}
\end{equation}
\noindent
where $\pi_W: \blowup_W X \xrightarrow{} X$ is the blow-up of $X$ along the subscheme defined by the entries of $\zeta_W$, $\Pi_W : \blowup_W X \xrightarrow{} Y_W$ is the projection onto the second factor, 
and $\pi_{L}:\P(W^\vee)\dashrightarrow \P_\C^3$ is the linear projection from some suitable subspace $L\subset \P (W^\vee)$. If $N=3$, then $\P (W^\vee) \cong \P_\C^3$ and $\pi_L$ is just an automorphism of $\P_\C^3$. 
It follows from Diagram \eqref{eq: commutative diagram} that $\phi$ is birational if and only if both $\zeta_W$ and the restriction of $\pi_L$ to $Y_W$ are birational. 
On the other hand, given a point 
\begin{equation}
\label{eq: point Q in X}
Q=(\alpha_0:\alpha_1)\times (\beta_0:\beta_1)\times (\gamma_0:\gamma_1)
\in 
X
\end{equation} 
we define the polynomials in $R$
$$
\Delta_1^Q \coloneqq 
\begin{vmatrix}
x_0 & x_1 \\
\alpha_0 & \alpha_1
\end{vmatrix}
\ ,\ 
\Delta_2^Q \coloneqq  
\begin{vmatrix}
y_0 & y_1 \\
\beta_0 & \beta_1
\end{vmatrix}
\ ,\ 
\Delta_3^Q \coloneqq  
\begin{vmatrix}
z_0 & z_1 \\
\gamma_0 & \gamma_1
\end{vmatrix}
\ .
$$
Before making specific choices for the vector space $W$, we first prove the following result.

\begin{lemma}
\label{lemma: condition on 0dim base locus}
Let $\phi$ be birational of type $(2,2,2)$. Then, there exists a unique point $Q$, as in \eqref{eq: point Q in X}, 
such that for every $0 \leq l \leq 3$ the entries of $\phi$ satisfy
\begin{equation}
\label{eq: condition on 0dim base locus}
f_l
\in
\left(
\lambda
\
\Delta_1^Q
\, y_j \, z_k
+
\mu
\
x_i \, 
\Delta_2^Q
\, z_k
+
\nu
\
x_i \, y_j \, 
\Delta_3^Q
\right)
+
\left(
\Delta_1^Q,\Delta_2^Q,\Delta_3^Q
\right)^2
\end{equation}
for some $\lambda,\mu,\nu\in\C^*$ and indices $0\leq i,j,k \leq 1$ where $\alpha_i,\beta_j,\gamma_k$ are all non-zero.
\end{lemma}

\begin{proof}
Consider the Segre embedding 
$ \sigma : X \hookrightarrow \Sigma \subset \P_\C^7$ of $X$, 
where the Segre variety $\Sigma$ is smooth of degree 6. From this isomorphism, it follows that the Hilbert series of $\Sigma$ is 
$$
\HS_\Sigma(t)=\sum_{d\geq 0} \dim_\C \Gamma(X,\Oc_X(d,d,d))\ t^d=\sum_{d\geq 0} (d+1)^3 \, t^d=\frac{1+4t+t^2}{(1-t)^4}\ ,
$$ 
where $(1-t)^4\, \HS_\Sigma(t)$ is a polynomial because $\dim \Sigma=3$. 
On the other hand, given a general line $\ell\subset \P_\C^3$  the image $\phi^{-1}(\ell)$ is a closed, irreducible curve in $X$. 
Moreover, if $(\sum_{i=0}^3 \alpha_i \, t_i, \sum_{i=0}^3 \beta_i \, t_i )$ is the defining ideal of $\ell$, then the ideal
$$
I' = \left(\sum_{i=0}^3 \alpha_i \, f_i, \sum_{i=0}^3 \beta_i \, f_i \right)
$$
defines $\phi^{-1}(\ell)$. More specifically, $I'$ is the ideal of  the scheme-theoretic union of $\phi^{-1}(\ell)$ and the base locus $Z_\phi$. As $\text{codim}\, I' = 2$ and $R$ is a Cohen-Macaulay ring, all the associated primes to $I'$ have codimension two. However we have $\dim Z_\phi = 0$, so $\text{codim}\, B_\phi = 3$. Therefore, $I'$ must be the defining ideal of $\phi^{-1}(\ell)$. In particular, 
the curve $\sigma(\phi^{-1}(\ell))$ is determined by two hyperplane sections in $\P_\C^7$, and its Hilbert series \cite[Chapter 5]{SingularBook} is then 
$$ \HS_{ \sigma(\phi^{-1}(\ell)) }(t) = \HS_{ \phi^{-1}(\ell) }(t) = (1-t)^2 \, \HS_\Sigma(t) =\frac{1+4t+t^2}{(1-t)^2}
\ ,
$$
implying that the Hilbert polynomial of $\phi^{-1}(\ell)$ is $\HP_{\phi^{-1}(\ell)}(t) = 6\, t$. In particular, the arithmetic genus of $\phi^{-1}(\ell)$ is one. By definition $\phi^{-1}(\ell)$ is a rational curve and its normalization is $\P_\C^1$. Then, from \cite[Chapter 4, Exercise 1.8]{HartshorneBook} $\phi^{-1}(\ell)$ has a unique singular point $Q$. As $\Sigma$ is smooth, by the strong Bertini's theorem \cite[Theorem 0.5]{EiHa16} the curve $\phi^{-1}(\ell)$ is smooth outside $Z_\phi$. Therefore, $Q$ is supported in $Z_\phi$.
\vskip0.2cm
Without loss of generality, we can assume that $Q$ belongs to the affine subset $U\subset X$ given by $x_0\not = 0,$ $y_0\not = 0,$ and $z_0\not = 0$, i.e$.$ we can assume indices $(i,j,k) = (0,0,0)$ in the statement. From the Jacobian criterion \cite[§$1$ Exercise $5.8$]{HartshorneBook} in $U$, as $Q$ is a singular point of $\phi^{-1}(\ell)$ the rank of the Jacobian matrix 
\begin{gather*}
    \text{Jac}_{\phi^{-1}(\ell)} =
    \begin{pmatrix}
    \sum_{l = 0}^3 \alpha_l\, \partial f_l / \partial x_1 & 
    \sum_{l = 0}^3 \alpha_l\, \partial f_l / \partial y_1 & 
    \sum_{l = 0}^3 \alpha_l\, \partial f_l / \partial z_1 \\ 
    \sum_{l = 0}^3 \beta_l \,\partial f_l / \partial x_1 & 
    \sum_{l = 0}^3 \beta_l \, \partial f_l / \partial y_1 & 
    \sum_{l = 0}^3 \beta_l\, \partial f_l / \partial z_1
    \end{pmatrix}
\end{gather*}
evaluated at $Q$ is one. Equivalently, writing $\nabla f_l = (\partial f_l / \partial x_1 \, , \, \partial f_l / \partial y_1 \, , \, \partial f_l / \partial z_1)$ then  
$$
\sum_{l = 0}^3 \alpha_l \, \nabla f_l(Q) \  ,\ 
\sum_{l = 0}^3 \beta_l \, \nabla f_l(Q)  
$$ 
must be proportional vectors.  
This is true for any choice of a general $\ell\subset \P_\C^3$, i.e$.$ for general values of the $\alpha_l$'s and $\beta_l$'s. Therefore, it can only occur if the $\nabla f_l(Q)$'s are all proportional to a vector $(\lambda,\mu,\nu)\in\C^3$. In particular, from the Taylor expansion of $f_l$ centered at $Q$ we can write
$$
f_l = \delta_l \, 
(\lambda \, \Delta_1^Q \, y_0 \, z_0 +
\mu \, x_0 \, \Delta_2^Q \, z_0 + 
\nu \, x_0 \, y_0 \, \Delta_3^Q
)
+
h_l
\ ,
$$
for some $\delta_l\in\C$ and $h_l\in (\Delta_1^Q, \Delta_2^Q, \Delta_3^Q )^2$. Moreover, if $\lambda = 0$ (resp$.$ $\mu = 0$ or $\nu = 0$) we find a line in the base locus $Z_\phi$, against $\dim Z_\phi = 0$. Hence, $\lambda,\mu,\nu$ are all non-zero. 
\end{proof}

\begin{remark}
Recall that a point $Q$ in the base locus $Z_\phi$ is a \textit{point of contact} to a surface $S\subset X$ smooth at $Q$ if $B_\phi \subset I(S) + I(Q)^2$. 
Geometrically, if $\phi$ is birational of type $(2,2,2)$, Lemma \ref{lemma: condition on 0dim base locus} establishes that every surface in the linear system spanned by the $f_l$'s has $Q$ as a point of contact to the surface given by
$$
\lambda
\
\Delta_1^Q
\, y_j \, z_k
+
\mu
\
x_i \, 
\Delta_2^Q
\, z_k
+
\nu
\
x_i \, y_j \, 
\Delta_3^Q
=
0
\ .
$$
\end{remark}

\begin{setup}
\label{setup: I_phi}
Given a point $Q$ as in \eqref{eq: point Q in X} and $\lambda,\mu,\nu\in\C^*$, we define the ideal in $R$
$$
I_{Q(\lambda,\mu,\nu)}
\coloneqq
\left(
\lambda
\
\Delta_1^Q
\, y_j \, z_k
+
\mu
\
x_i \, 
\Delta_2^Q
\, z_k
+
\nu
\
x_i \, y_j \,
\Delta_3^Q
\right)
+
\left(
\Delta_1^Q
\, , \, 
\Delta_2^Q
\, , \, 
\Delta_3^Q
\right)^2
$$
for the smallest indices $0\leq i,j,k \leq 1$ such that $\alpha_i,\beta_j,\gamma_k$ are all non-zero. The graded component of $I_{Q(\lambda,\mu,\nu)}$ in tri-degree $(1,1,1)$ is generated by the five independent polynomials
\begin{gather*}
\lambda
\
\Delta_1^Q
\, y_j \, z_k
+
\mu
\
x_i \, 
\Delta_2^Q
\, z_k
+
\nu
\
x_i \, y_j \,
\Delta_3^Q
\ ,
\\
\Delta_1^Q \, \Delta_2^Q \, z_k
\ ,\ 
\Delta_1^Q \, y_j \, \Delta_3^Q
\ ,\ 
x_i \, \Delta_2^Q \, \Delta_3^Q
\ ,\ 
\Delta_1^Q \, \Delta_2^Q \, \Delta_3^Q
\ .
\end{gather*}
In particular, there is always an automorphism of $R$ preserving the tri-graded structure (equivalently, an automorphism of $X$) that sends  $I_{Q(\lambda,\mu,\nu)}$ into the ideal $I_{Q'(1,1,1)}$, where $Q' = (1:0)^3$. 
\end{setup}

\vskip0.2cm

The following is a useful description of the Cohen-Macaulay curves in $X$ of the tri-degrees of interest for the later analysis.

\begin{lemma}
\label{lemma: canonical CM curves}
Let $C$ be a Cohen-Macaulay curve in $X$. Then, we have the following:
\begin{itemize}
\setlength\itemsep{0.5em}
\item {If $C$ has tri-degree $(1,1,1)$ there is an automorphism of $X$ sending it to the curve defined by one of the following ideals:
\vskip-0.2cm
\begin{minipage}[t]{0.45\textwidth}
\begin{enumerate}
\item $(x_0y_1-x_1y_0,x_0z_1-x_1z_0,y_0z_1-y_1z_0)$
\item $(x_0y_1-x_1y_0,z_1)\cap (x_1,y_0)$
\item $(x_0y_1-x_1y_0,z_1)\cap (x_1,y_1)$
\item $(x_1,y_1)\cap (x_0,z_1)\cap (y_0,z_0)$
\end{enumerate}
\end{minipage}
\begin{minipage}[t]{0.45\textwidth}
\begin{enumerate}
\setcounter{enumi}{4}
\item $(x_1,y_1)\cap (x_1,z_1)\cap (y_0,z_0)$
\item $(x_1,y_1)\cap (x_1,z_1)\cap (y_1,z_0)$
\item $(x_1,y_1)\cap (x_1,z_1)\cap (y_1,z_1)$
\end{enumerate}
\end{minipage}
}
\item {If $C$ has tri-degree either $(1,1,0)$, $(1,0,1)$, or $(0,1,1)$ there is an automorphism of $X$ sending it to the curve defined by one of the following ideals:
\vskip-0.2cm
\begin{minipage}[t]{0.45\textwidth}
\begin{enumerate}
\setcounter{enumi}{7}
\item $(x_0y_1-x_1y_0,z_1)$
\item $(x_1,z_0)\cap (y_1,z_1)$
\end{enumerate}
\end{minipage}
\begin{minipage}[t]{0.45\textwidth}
\begin{enumerate}
\setcounter{enumi}{9}
\item $(x_1,z_1)\cap (y_1,z_1)$
\end{enumerate}
\end{minipage}
}
\item {If $C$ has tri-degree either $(1,0,0)$, $(0,1,0)$, or $(0,0,1)$ it is a projective line, and it can be transformed by an automorphism of $X$ into the line of ideal $(y_1,z_1)$.}
\end{itemize}
\end{lemma}
\begin{proof}
Up to an automorphism of $X$ that simply permutes the factors (see §\ref{subsection: automorphisms of X}), we can assume that $C$ has tri-degree either $(1,1,1)$, $(1,1,0)$, or $(1,0,0)$. In the latter case, the projection $\pi_1:C\xrightarrow{} \P_\C^1$ onto the first factor of $X$ is an isomorphism, so $C$ is a projective line. In particular, there is an automorphism of $X$ sending $C$ to the line $(y_1,z_1)$.
\vskip0.2cm
Now, we assume that $C$ is irreducible of tri-degree either $(1,1,1)$ or $(1,1,0)$. The projections $\pi_i:X\xrightarrow{}\P_\C^1$ onto the first two factors of $X$ have degree one, so $C$ is rational. Hence, $C$ is the image of a regular map $\varphi: \P_\C^1 \xrightarrow{} C$ given by
\begin{equation}
\label{eq: parametrization of CM curve}
    (t_0:t_1) \mapsto (u_0(t_0,t_1):u_1(t_0,t_1))\times
    (v_0(t_0,t_1):v_1(t_0,t_1))\times
    (w_0(t_0,t_1):w_1(t_0,t_1))\ ,
\end{equation}
for some linear forms $u_i,v_j,w_k\in \C[t_0,t_1]$ (if $C$ has tri-degree $(1,1,0)$, the $w_k$'s are proportional and we have $(w_0(t_0,t_1):w_1(t_0,t_1)) = (\gamma_0:\gamma_1)$ for some $(\gamma_0:\gamma_1)\in\P_\C^1$). Therefore, any two parametrizations of curves of the same tri-degree coincide up to the action of $\PGL(2,\C)^3 \leq \Aut(X)$. In particular, if $C$ has tri-degree $(1,1,1)$ (resp$.\, (1,1,0)$) there is an automorphism of $X$ sending it to the curve defined by the ideal $(1)$ (resp$.\, (8)$) in the statement. 
\vskip0.2cm
Now, assume that $C$ has exactly two irreducible components of tri-degrees $(1,1,0)$ and $(0,0,1)$. By the argument above, the component of tri-degree $(1,1,0)$ can be transformed into the curve $D\subset X$ defined by the ideal $(x_0 y_1 - x_1 y_0 , z_1)$. Moreover, the component of tri-degree $(0,0,1)$ can be parametrized as
\begin{equation*}
    (t_0:t_1) \mapsto ( \alpha_0:\alpha_1 )\times
    ( \beta_0:\beta_1 )\times
    (w_0'(t_0,t_1):w_1'(t_0,t_1))\ ,
\end{equation*}
for some $(\alpha_0:\alpha_1)\times (\beta_0:\beta_1) \in \P_\C^1\times \P_\C^1$ and linear forms $w_0',w_1'$ in $\C [t_0,t_1]$. In particular, if $(\alpha_0:\alpha_1)\times (\beta_0:\beta_1) \in \P_\C^1\times \P_\C^1$ satisfies $\alpha_0 \beta_1-\alpha_1 \beta_0 = 0$ we find an automorphism of $X$ that stabilizes $D$ and sends
$$
( \alpha_0:\alpha_1 )\times
( \beta_0:\beta_1 )\times
( w_0'(t_0,t_1):w_1'(t_0,t_1) )
\mapsto
(1:0)\times
(1:0)\times
( w_0'(t_0,t_1):w_1'(t_0,t_1) )
\ .
$$
On the other hand, if $\alpha_0 \beta_1-\alpha_1 \beta_0 \not= 0$ then we find an automorphism sending
$$
( \alpha_0:\alpha_1 )\times
( \beta_0:\beta_1 )\times
( w_0'(t_0,t_1):w_1'(t_0,t_1) )
\mapsto
(1:0)\times
(0:1)\times
( w_0'(t_0,t_1):w_1'(t_0,t_1) )
$$
that stabilizes $D$. The images of the two last parametrizations are the projective lines in $X$ respectively defined by $(x_1,y_1)$ and $(x_1,y_0)$. Therefore, $C$ can be transformed into the curve defined by either the ideal $(3)$ or $(2)$ in the statement by means of an automorphism of $X$. Proceeding similarly with curves of tri-degree $(1,1,1)$ with three irreducible components, one derives the curves from $(4)$ to $(7)$. On the other hand, if $C$ is reducible of tri-degree $(1,1,0)$ it must have exactly two irreducible components of tri-degrees $(1,0,0)$ and $(0,1,0)$. Depending on whether these components intersect, $C$ can be transformed into the curve defined by either $(9)$ or $(10)$.
\end{proof}

Recall from §\ref{subsection: curves in X} that if $\dim Z_\phi = 1$ then $I_\phi$ stands for the intersection of all the associated primes to $B_\phi$ in $R$ of codimension exactly two. On the other hand,  if $\dim Z_\phi = 0$ we define $I_\phi \coloneqq I_{Q(\lambda,\mu,\nu)}$ where $Q\in X$ and $\lambda,\mu,\nu\in\C^*$ are as in the statement of Lemma \ref{lemma: condition on 0dim base locus}. Additionally, in all the cases we define $W_\phi \coloneqq (I_\phi)_{(1,1,1)}$. In particular, if $W = W_\phi$ then \eqref{eq: vector space W} is satisfied. Moreover, vectors in $W_\phi$ are identified to surfaces in the linear system that the $f_i$'s define in $X$.

\vskip0.2cm

The following result provides useful information about   Diagram \eqref{eq: commutative diagram}, according to the type of $\phi$. 
Recall that a subvariety $V\subset \P_\C^N$ has \textit{minimal degree} if it is non-degenerate and $\deg V = 1 + \text{codim }V$.

\begin{proposition}
\label{proposition: commutative diagram}
Assume that $\phi$ is birational, and set $W = W_\phi$. We have the following:
\begin{itemize}
    \item If $\phi$ has type $(1,1,1)$, then $N = 3$ and $Y_{ W } \cong \P_\C^3$.
    \item If $\phi$ has type $(1,1,2)$, then $N = 4$ and $\deg Y_{ W } = 2$.
    \item If $\phi$ has type $(1,2,2)$, then $N = 5$ and $\deg Y_{ W } = 3$.
    \item If $\phi$ has type $(2,2,2)$, then $N = 4$ and $\deg Y_{ W } = 2$.
\end{itemize}
In all the cases, $Y_W$ has minimal degree and $\zeta_W$ is birational for any choice of basis in $W^\vee$. 
Moreover, if $\dim Z_\phi = 1$ the curve $C_\phi$ is connected.
\end{proposition}

\begin{proof}
We prove each statement separately:
\begin{itemize}
\setlength\itemsep{0.3em}
\item If $\phi$ has type $(1,1,1)$, the base locus $Z_\phi$ has tri-degree $(1,1,1)$. Therefore, up to an automorphism of $X$ the curve $C_\phi$ is one of the seven listed in Lemma \ref{lemma: canonical CM curves}. The ideals $(1),(3),(6)$, and $(7)$ determine connected curves, and satisfy $\dim_\C W = 4$. On the other hand, the ideals $(2),(4),(5)$ yield $\dim_\C W = 3$. As the entries $f_l$'s of $\phi$ are independent and they belong to $W$, we must have $\dim_\C W \geq 4$. Hence, $\dim_\C W = 4$ and $C_\phi$ is a connected curve.
\item If $\phi$ has type $(1,1,2)$, the base locus $Z_\phi$ has tri-degree $(1,1,0)$. Therefore, up to an automorphism of $X$ the curve $C_\phi$ is one of the three listed in Lemma \ref{lemma: canonical CM curves}. The ideals $(8)$ and $(10)$ determine connected curves and satisfy $\dim_\C W = 5$. On the other hand, the ideal $(9)$ defines two skew lines. Moreover, the ideal generated by the polynomials of tri-degree $(1,1,1)$ in $(9)$ defines a curve of tri-degree $(1,1,1)$ in $X$. From \eqref{eq: vector space W}, it follows that if $C_\phi$ is equivalent to $(9)$ then $Z_\phi$ has tri-degree $(1,1,1)$. 
Therefore, we must have $\dim_\C W = 5$ and $C_\phi$ is again a connected curve. The degree of $Y_W$ follows from a direct computation (performed with the help of \textsc{Macaulay2}) with the ideals $(8)$ and $(10)$.
\item If $\phi$ has type $(1,2,2)$, the base locus $Z_\phi$ has tri-degree $(1,0,0)$. By Lemma \ref{lemma: canonical CM curves} the curve $C_\phi$ is a projective line, so it is connected and $\dim_\C W = 6$. Similarly, the degree of $Y_W$ follows from a direct computation on the ideal $(y_0,z_0)$.
\item If $\phi$ has type $(2,2,2)$, it is immediate from Setup \ref{setup: I_phi}
that $\dim_\C W=5$. Similarly, it follows from a direct computation that $\deg Y_W = 2$.
\end{itemize}
\vskip0.2cm
On the other hand, two distinct basis in $W^\vee$ yield rational maps $\zeta_W$ that coincide up to an automorphism of $\P_\C^N$. The birationality of $\zeta_W$ then follows from the explicit construction of this rational map when $W$ is the component in tri-degree $(1,1,1)$ of the ideals listed above.
\end{proof}

\section{The algebraic set of tri-linear birational maps}
\label{section: the algebraic set of tri-linear birational maps}

We now describe the algebraic set $\Bir_{(1,1,1)}$ of tri-linear birational maps after quotient by the equivalence relation given by composition with an automorphism of $\P_\C^3$. Additionally, we introduce the group action of $\Aut(X)$ on $\Bir_{(1,1,1)}$. 

\subsection{$\Bir_{(1,1,1)}$ and $\Birr_{(r_1,r_2,r_3)}$} 
\label{subsection: Bir(1,1,1) and Bir(r1,r2,r3)}
We write $\Rat_{(1,1,1)}$ for the set of ordered tuples of four tri-linear polynomials in the variables of $X$ without a common factor, up to non-zero scalar. The elements in $\Rat_{(1,1,1)}$ are identified with tri-linear rational maps. 
In this set, we introduce the equivalence relation $\sim$ given by setting $\phi \sim \phi'$ if $W_{\phi} = W_{\phi'}$. Equivalently, $\phi\sim \phi'$ if there is $\rho\in\Aut(\P_\C^3)$  such that $\phi' = \rho\circ\phi$. 
\vskip0.2cm
We define $\Bir_{(1,1,1)}$ as the quotient by $\sim$ of the subset of $\Rat_{(1,1,1)}$ given by the  tuples associated to birational maps. 
In particular, if $V_{(1,1,1)}$ is the $\C$-vector space of tri-linear polynomials in the variables of $X$, 
there is an injective map of sets from $\Bir_{(1,1,1)}$ to the Grassmannian $\mathbb{G}(4,V_{(1,1,1)})$ given by $[\phi]\mapsto W_\phi$. Therefore, we can write $W_\phi$ instead of $[\phi]$.
\vskip0.2cm
Clearly, two birational maps $\phi$ and $\phi'$ in the same class of $\Bir_{(1,1,1)}$ have the same type.  Thus, given $1\leq r_1,r_2,r_3 \leq 2$ it makes sense to define the sets
$$
\Birr_{(r_1,r_2,r_3)} \coloneqq \{ W_\phi\in \Bir_{(1,1,1)}:\, \phi \text{ has type }(r_1,r_2,r_3)
\}\ .
$$
From \textsc{Table} \ref{table: possible types of inverse}, these subsets form a partition of $\Bir_{(1,1,1)}$. 

\subsection{The group action of $\Aut(X)$}

The map
\begin{eqnarray}
\label{eq: group action of automorphisms of X in Bir_(1,1,1)}
	\Aut(X)\times \Bir_{(1,1,1)} & \xrightarrow{} & \Bir_{(1,1,1)} \\ \nonumber
    \xi \times W_\phi & \mapsto & W_{\phi\circ \xi}
\end{eqnarray}
is a well-defined group action of $\Aut(X)$ on $\Bir_{(1,1,1)}$, that we call \textit{right-action}. 
Namely, given birational maps $\phi$ and $\phi'$, the vector spaces $W_\phi$ and $W_{\phi'}$ belong to the same orbit of the right-action if and only if there are $\rho\in \Aut(\P_\C^3)$ and $\xi\in \Aut(X)$ such that $\phi' = \rho \circ \phi \circ \xi$. With a small abuse in the notation, we also say that $\phi$ and $\phi'$ belong to the same orbit. In particular, the base loci of two birational maps in the same orbit coincide up to automorphism of $X$.

\subsection{The algebraic set $\Bir_{(1,1,1)}$}

We write $\mathscr{S}(Y_W)$ for the variety of $(N - 4)$-dimensional subspaces secant to $Y_W$. Now, we prove the main result of this section.
\begin{theorem}
\label{theorem: algebraic structure of Bir(1,1,1)}
Let $V_{(1,1,1)}$ be the $\C$-vector space of tri-linear polynomials in the variables of $X$. 
Then, the set $\Bir_{(1,1,1)}$ has the structure of a locally closed algebraic subset of the Grassmannian $\mathbb{G}(4,V_{(1,1,1)})$. Moreover, its irreducible components are the following:
\begin{itemize}
    \item  $\Birr_{(1,1,1)}$, of dimension 6.
    \item $\Birr_{(1,1,2)}$,  $\Birr_{(1,2,1)}$, and $\Birr_{(2,1,1)}$, all of dimension $7$.
    \item $\Birr_{(1,2,2)}$, $\Birr_{(2,1,2)}$, and $\Birr_{(2,2,1)}$, all of dimension $8$.
    \item $\Birr_{(2,2,2)}$, of dimension $8$.
\end{itemize}
\end{theorem}

\begin{proof}
The birational map
\begin{eqnarray}
\label{eq: birational map; Bir is algebraic set}
	\Phi: X & \dasharrow & \P_\C^3 \\ \nonumber
	(x_0:x_1)\times(y_0:y_1)\times(z_0:z_1) & \mapsto & (x_0y_0z_0:x_1y_0z_0:x_0y_1z_0:x_0y_0z_1)
\end{eqnarray}
is an isomorphism between the affine subset  $U\subset X$ determined by $x_0\not= 0, y_0\not= 0,z_0\not= 0$ and the affine subset of $\P_\C^3 = \Proj(\C[t_0,t_1,t_2,t_3])$ given by $t_0\not= 0$. Write $V_3$ for the $\C$-vector space of cubic forms in $t_0$,$t_1$,$t_2$, and $t_3$. We have the injective linear map $\chi: V_{(1,1,1)} \xrightarrow{} V_3$ given by
\begin{gather*}
x_0y_0z_0 \mapsto t_0^3
\ ,\ 
x_1y_0z_0 \mapsto t_0^2t_1
\ ,\ 
x_0y_1z_0 \mapsto t_0^2t_2
\ ,\ 
x_0y_0z_1 \mapsto t_0^2t_3
\ ,\\ 
x_1y_1z_0 \mapsto t_0t_1t_2
\ ,\ 
x_1y_0z_1 \mapsto t_0t_1t_3
\ ,\ 
x_0y_1z_1 \mapsto t_0t_2t_3
\ ,\ 
x_1y_1z_1 \mapsto t_1t_2t_3
\ .
\end{gather*}
In particular, the rational endomorphism $\psi: \P_\C^3 \dasharrow \P_\C^3$ that sends
$$
(t_0:t_1:t_2:t_3)\mapsto (\chi(f_0):\chi(f_1):\chi(f_2):\chi(f_3))
$$  
coincides with $\phi\circ\Phi^{-1}$ on the open set $\Phi(U)$. Therefore, the association 
$
\phi\mapsto \phi\circ \Phi^{-1}
$
is an identification between tri-linear rational maps and rational endomorphisms of $\P_\C^3$ whose entries belong to the vector subspace
$$
V' = 
\langle t_0^3 \, ,\, t_0^2 t_1  \, ,\, t_0^2 t_2  \, ,\, t_0^2 t_3  \, ,\, t_0 t_1 t_2  \, ,\, t_0 t_1 t_3  \, ,\, t_0 t_2 t_3  \, ,\, t_1 t_2 t_3 \rangle
\subset
V_3
\ .
$$
By \cite[Proposition B]{transformationsquadratiques}, the classes (up to composition with an automorphism of $\P_\C^3$) of birational automorphisms of $\P_\C^3$ with cubic entries form a locally closed algebraic set $Y'\subset\mathbb{G}(4,V_3)$. On the other hand, the Grassmannian $\mathbb{G}(4,V')$ is a subvariety of $\mathbb{G}(4,V_3)$. Hence, the intersection
$$
\Bir_{(1,1,1)} = Y' \cap \mathbb{G}(4,V')
$$
is also a locally closed algebraic subset of the Grassmannian $\mathbb{G}(4,V')\cong\mathbb{G}(4,V_{(1,1,1)})$.
\vskip0.4cm
Now, we prove the statements about the irreducible components of $\Bir_{(1,1,1)}$. Recall from Proposition \ref{proposition: commutative diagram} that if $\dim Z_\phi = 1$ the Cohen-Macaulay curve $C_\phi$ is connected. 
Let $\Cc_{(1,1,1)}$ be the variety of connected Cohen-Macaulay curves of tri-degree $(1,1,1)$, which is irreducible of dimension 6 \cite[Proposition 4.9]{ballico}.
Additionally, consider the algebraic set
$$
\mathscr{V}_{(1,1,1)} = \{ W_\phi \times C \in \Bir_{(1,1,1)}\times \Cc_{(1,1,1)} : W_\phi \subset {(I_C)}_{(1,1,1)} \}\ ,
$$
together with the canonical projections $\Pi_1: \mathscr{V}_{(1,1,1)} \xrightarrow{} \Bir_{(1,1,1)}$ and $\Pi_2: \mathscr{V}_{(1,1,1)} \xrightarrow{} \Cc_{(1,1,1)}$. 
Given a curve $C\in \Cc_{(1,1,1)}$ we have $\Pi_2^{-1}(C) = W \times C$, where $W = {(I_C)}_{(1,1,1)}$ as $\dim_\C W = 4$ by  Proposition \ref{proposition: commutative diagram}. Hence, $\mathscr{V}_{(1,1,1)}$ is irreducible of dimension $6$. On the other hand, given $W_\phi\in \Bir_{(1,1,1)}$ we have $\Pi_1^{-1}(W_\phi) = W_\phi \times C_\phi$ if $\phi$ has type $(1,1,1)$ and $\Pi_1^{-1}(W_\phi) = \varnothing$ otherwise. Therefore, it follows that $\Birr_{(1,1,1)} \cong \Cc_{(1,1,1)}$.
\vskip0.2cm
The study of the irreducible components $\Birr_{(1,1,2)}$, $\Birr_{(1,2,1)}$, and $\Birr_{(2,1,1)}$ is identical, as the associated birational maps differ by a permutation of the factors of $X$. Therefore, we can focus on $\Birr_{(1,1,2)}$. By Lemma \ref{lemma: canonical CM curves}, the ideal of a connected Cohen-Macaulay curve of tri-degree $(1,1,0)$ is of the form
$$
(\alpha_0 \, x_0 y_0 + \alpha_1 \, x_1 y_0 + \alpha_2 \, x_0 y_1 + \alpha_3 \, x_1 y_1,\gamma_0 z_1 - \gamma_1 z_0)\ ,
$$
for some $(\alpha_0 : \alpha_1 : \alpha_2 : \alpha_3)\times (\gamma_0 : \gamma_1 ) \in \P_\C^3 \times \P_\C^1$, where the first generator might be reducible. Thus, we can identify the variety $\Cc_{(1,1,0)}$ of such curves with $\P_\C^3 \times \P_\C^1$. Consider now the algebraic set
$$
\mathscr{V}_{(1,1,2)} = \{ W_\phi \times C \in \Bir_{(1,1,1)}\times \Cc_{(1,1,0)} : W_\phi \subset {(I_C)}_{(1,1,1)} \}\ ,
$$
together with the projections as before. Given a curve $C \in \Cc_{(1,1,0)}$, by \cite[Proposition A]{transformationsquadratiques} the fiber $\Pi_2^{-1}(C)$ is isomorphic to an open subset of $\mathscr{S}(Y_W)$, where $W = {(I_C)}_{(1,1,1)}$. Hence, $\mathscr{V}_{(1,1,2)}$ is irreducible of dimension $\dim \Cc_{(1,1,0)} + \dim \mathscr{S}(Y_W) = 4 + 3 = 7$. On the other hand, we have the identity of sets
$$
\Birr_{(1,1,2)} = \Pi_1 ( \mathscr{V}_{(1,1,2)} ) \, \backslash \, \Birr_{(1,1,1)}\ ,
$$
so given a general $W_\phi\in \Pi_1 ( \mathscr{V}_{(1,1,2)} )$ the base locus $Z_\phi$ has tri-degree $(1,1,0)$. Thus, we have $\Pi_1^{-1}(W_\phi) = (W_\phi,C_\phi)$ and $\Pi_1$ is  birational, 
implying that $\Birr_{(1,1,2)}$ is irreducible of dimension 7. Similarly, one proves that $\Birr_{(1,2,2)}$, $\Birr_{(2,1,2)}$, and $\Birr_{(2,2,1)}$ are all irreducible of dimension 8. 
\vskip0.2cm
For the irreducible component $\Birr_{(2,2,2)}$, we consider the algebraic set
$$
\mathscr{V}_{(2,2,2)} = \{ W_\phi \times Q \times (\lambda,\mu) \in \Bir_{(1,1,1)}\times X \times (\C^*)^2 : W_\phi \subset {(I_{Q(\lambda,\mu,1)})}_{(1,1,1)} \}\ ,
$$
with the projections $\Pi_1 : \mathscr{V}_{(2,2,2)} \xrightarrow{} \Bir_{(1,1,1)}$ and $\Pi_2 : \mathscr{V}_{(2,2,2)} \xrightarrow{} X\times (\C^*)^2$. By \cite[Proposition A]{transformationsquadratiques}, given $Q \times (\lambda,\mu) \in X\times (\C^*)^2$ the fiber $\Pi_2^{-1}( Q \times (\lambda,\mu) )$ is isomorphic to an open subset of $\mathscr{S}(Y_W)$, where $W = {(I_{Q(\mu,\nu,1)})}_{(1,1,1)}$.  
Hence, $\mathscr{V}_{(2,2,2)}$ is irreducible of dimension 
$$
\dim X + \dim (\C^*)^2 + \dim \mathscr{S}(Y_W) = 3 + 2 + 3 = 8 \ .
$$ 
On the other hand, given a general $W_\phi\in \Pi_1( \mathscr{V}_{(2,2,2)} )$ we have $\Pi_1^{-1}(W_\phi) = W_\phi \times Q \times (\lambda,\mu)$ for the unique singular point $Q\in Z_\phi$ and some $(\lambda,\mu)\in (\C^*)^2$. In particular, $\Pi_1$ is birational. Moreover, as a general point in $\Pi_1(  \mathscr{V}_{(2,2,2)} )$ belongs to $\Birr_{(2,2,2)}$ the latter is irreducible of dimension 8.
\end{proof}

\section{Geometric classification of tri-linear birational maps} 
\label{section: classification}

In this section, we study the right-action on $\Bir_{(1,1,1)}$ and give the complete list of its orbits. More specifically we find exactly 19 orbits, and each of these determines an isomorphism class of the possible base loci of a tri-linear birational map.
\vskip0.2cm
Recall that a point $Q$ in the base locus $Z_\phi$ is a \textit{point of contact} to a surface $S\subset X$ smooth at $Q$ if $B_\phi \subset I(S) + I(Q)^2$. Similarly, $Q$ is a \textit{point of tangency} to a curve $C\subset X$ smooth at $Q$ if $B_\phi \subset I(C) + I(Q)^2$. On the other hand, when we refer to ``a surface in $W_\phi$'' we mean a surface in $X$ whose equation belongs to the linear system $W_\phi$ spanned by the entries of $\phi$.

\subsection{Orbits of birational maps of type $(1,1,1)$}
\label{subsection: (1,1,1) orbits}
The classification of the orbits in $\Birr_{(1,1,1)}$ is the most straightforward, as from Proposition \ref{proposition: syzygies (1,1,1)} (that we postpone until §\ref{section: syzygy-based characterization}) the base locus is a Cohen-Macaulay curve.  Namely, we have the following.
\begin{theorem}
\label{theorem: classification (1,1,1)}
Let $\phi$ be birational of type $(1,1,1)$. Then, $\phi$ belongs to the orbit of one of the following birational maps:
\begin{itemize}
    \item $\rho_1^{(1,1,1)} \equiv (x_1y_0z_1-x_0y_1z_1 : x_1y_1z_0-x_0y_1z_1: x_0y_1z_0-x_0y_0z_1: x_1y_0z_0-x_0y_0z_1)$
    \vskip0.0cm
    \noindent The entries of $\rho_1^{(1,1,1)}$ are given by the $3\times 3$ minors of the matrix
    $$
    \begin{pmatrix}
    0 & -y_1 & -z_1 \\
    -x_1 & 0 & z_1 \\
    x_0 & y_0 & 0 \\
    -x_0 & 0 & z_0 
    \end{pmatrix}\ .
    $$
    A surface in $W_{\rho_1}$ contains an irreducible curve of tri-degree $(1,1,1)$. The base ideal is 
    $$(x_0y_1-x_1y_0, x_0z_1-x_1z_0, y_0z_1-y_1z_0)\ .
    $$
    \item $\rho_2^{(1,1,1)} \equiv (x_1 y_1 z_1 : x_0 y_1 z_1 : x_1 y_0 z_1 : x_1 y_0 z_0 - x_0 y_1 z_0)$
    \vskip0.0cm
    \noindent The entries of $\rho_2^{(1,1,1)}$ are given by the $3\times 3$ minors of the matrix
    $$
    \begin{pmatrix}
    0 & 0 & -z_1 \\
    0 & -y_1 & z_0 \\
    -x_1 & 0 & -z_0 \\
    x_0 & y_0 & 0 
    \end{pmatrix}\ ,
    $$
    A surface in $W_{\rho_2}$ contains an irreducible curve of tri-degree $(1,1,0)$ and a line of tri-degree $(0,0,1)$, these two intersecting. The base ideal is 
    $$
    (x_1y_0-x_0y_1,z_1)\cap (x_1,y_1) \ .
    $$
    \item $\rho_{3}^{(1,1,1)} \equiv (x_1y_1z_1: x_0y_1z_1: x_1y_1z_0: x_1y_0z_0)$
    \vskip0.0cm
    \noindent The entries of $\rho_{3}^{(1,1,1)}$ are given by the $3\times 3$ minors of the matrix
    $$
    \begin{pmatrix}
    0 & -y_1 & 0 \\
    0 & y_0 & -z_1 \\
    -x_1 & 0 & 0 \\
    x_0 & 0 & z_0 
    \end{pmatrix}\ .
    $$
    A surface in $W_{\rho_3}$ contains three lines of tri-degrees $(1,0,0), (0,1,0),$ and $(0,0,1)$, one of them intersecting the other two at distinct points. The base ideal is 
    $$
    (x_1,y_1)\cap (x_1,z_1)\cap (y_1,z_0) \ .
    $$
    
    \item $\rho_{4}^{(1,1,1)} \equiv (x_1y_1z_1: x_0y_1z_1: x_1y_0z_1: x_1y_1z_0)$
    \vskip0.0cm
    \noindent The entries of $\rho_{4}^{(1,1,1)}$ are given by the $3\times 3$ minors of the matrix
    $$
    \begin{pmatrix}
    0 & 0 & -z_1 \\
    0 & -y_1 & 0 \\
    -x_1 & 0 & 0 \\
    x_0 & y_0 & z_0 
    \end{pmatrix}\ .
    $$
    A surface in $W_{\rho_4}$ contains three lines of tri-degrees $(1,0,0), (0,1,0),$ and $(0,0,1)$ that intersect at a common point. The base ideal is 
    $$
    (x_1,y_1)\cap (x_1,z_1)\cap (y_1,z_1) \ .
    $$
\end{itemize}
\end{theorem}

\begin{proof}
From Proposition \ref{eq:HBres}, the base locus of $\phi$ is a connected Cohen-Macaulay curve of tri-degree $(1,1,1)$, i.e$.$ we have $Z_\phi = C_\phi$. Additionally, from the proof of Proposition \ref{proposition: commutative diagram} given a connected Cohen-Macaulay curve $C$ of tri-degree $(1,1,1)$ we have $\dim_\C (I_C)_{(1,1,1)} = 4$. Hence, any such curve determines a unique $W_\phi$ in $\Bir_{(1,1,1)}$. On the other hand, by Lemma \ref{lemma: canonical CM curves} any connected Cohen-Macaulay curve of tri-degree $(1,1,1)$ is equivalent, by means of an automorphism of $X$, to the base loci of one of the birational maps in the statement.
\end{proof}

\subsection{Orbits of birational maps of type $(1,1,2)$, $(1,2,1)$ and $(2,1,1)$}
\label{subsection: (1,1,2) orbits} 
Let $C\subset X$ be a connected Cohen-Macaulay curve of tri-degree either $(1,1,0)$, $(1,0,1)$, or $(0,1,1)$. From Lemma \ref{lemma: canonical CM curves} and Proposition \ref{proposition: commutative diagram}, by means of an automorphism of $X$ the curve $C$ can be transformed into either $C_o$ or $C_\times$, respectively defined by the ideals $I_{C_o} = (x_0 y_1 - x_1 y_0, z_1)$ and $I_{C_\times} = (x_1 y_1 , z_1)$.
In particular, any birational map of type $(1,1,2)$, $(1,2,1)$, or $(2,1,1)$ lies in the orbit of a birational map with either $C_o$ or $C_\times$ in its base locus.
\vskip0.2cm
According to Diagram \eqref{eq: commutative diagram}, in order to study the orbits of the right-action in $\Birr_{(1,1,2)}$, $\Birr_{(1,2,1)}$ and $\Birr_{(2,1,1)}$ it is sufficient to compute the orbits of the group action given by the stabilizer $\text{Stab}(C_o) \leq \Aut(X)$ of the curve $C_o$ acting on $Y_{W_o}$, where $W_o \coloneqq {(I_{C_o})}_{(1,1,1)}$, and the orbits of the group action given by the stabilizer $\text{Stab}(C_\times) \leq \Aut(X)$ of the curve $C_\times$ acting on $Y_{W_\times}$, where $W_\times \coloneqq {(I_{C_\times})}_{(1,1,1)}$.

\subsubsection{Orbits in $Y_{W_o}$}
\label{subsection: orbits nonsingular (1,1,0)}
In this subsection, we set $W = W_o$. The birational map $\zeta_W$ appearing in Diagram \eqref{eq: commutative diagram} is 
$$
\zeta_{W}: (x_0:x_1)\times (y_0:y_1)\times (z_0:z_1) \mapsto ( x_0y_0z_1 : x_1y_0z_1 : x_0y_1z_1 : x_1y_1z_1 : (x_0y_1-x_1y_0)z_0 )\ ,
$$
and $Y_W\subset \P_\C^4$ is the cone determined by $t_0t_3 - t_1t_2 = 0$. Equivalently, 
$Y_W$ is the image of the projection $\Pi_W: \blowup_W X \subset X\times \P_\C^4 \xrightarrow{} \P_\C^4$, where $\pi_W : \blowup_W X \xrightarrow{} X$ is the blow-up of $X$ along $C_o$. The exceptional divisor $E_W$ of $\blowup_W X$ projects onto a two-dimensional cone in $Y_W$. Namely, the defining equation of $\Pi_W (E_W)$ in $Y_W$ is $t_1 - t_2 = 0$, and its vertex is the point $P=$ ($0\,$:$\,0\,$:$\,0\,$:$\,0\,$:$\,1$). 
We write $H_{1-2}$ and $H_4$ for the divisors in $Y_W$ determined by $t_1-t_2 = 0$ and $t_4 = 0$, respectively.
\begin{lemma}
\label{lemma: (1,1,2) orbits in Y_W irreducible}
The group action of $\text{Stab}(C_o)$ on $Y_W$ determines the four orbits:
$$
Y_W \backslash H_{1-2}\ ,\ H_{1-2}\backslash (H_{4}\cup P)\ ,\ H_{1-2}\cap H_4\ ,\ P
$$
\end{lemma}
\begin{proof}
Write $D_{xy}$ and $D_z$ for the divisors in $X$ defined by $x_0y_1-x_1y_0 = 0$ and $z_1 = 0$, respectively. The group action of $\text{Stab}(C_o)$ on $X$ determines the four orbits
\begin{equation}
\label{eq: orbits in X of stabilizer of (1,1,0)}
C_o = D_{xy}\cap D_z\ ,\ D_{z} \backslash C_o\ ,\ D_{xy} \backslash C_o\ ,\ X \backslash (D_{z}\cup D_{xy})\ .
\end{equation}
Namely, the action of an automorphism in $\text{Stab}(C_o)$ on the point $(x_0:x_1)\times (y_0:y_1) \times (z_0:z_1)$ is explicitly given by either
$$
\begin{pmatrix}
\lambda_{00} & \lambda_{01} \\
\lambda_{10} & \lambda_{11}
\end{pmatrix}
\begin{pmatrix}
x_0 \\
x_1
\end{pmatrix}
\times
\begin{pmatrix}
\lambda_{00} & \lambda_{01} \\
\lambda_{10} & \lambda_{11}
\end{pmatrix}
\begin{pmatrix}
y_0 \\
y_1
\end{pmatrix}
\times
\begin{pmatrix}
\mu_{00} & \mu_{01} \\
0 & \mu_{11}
\end{pmatrix}
\begin{pmatrix}
z_0 \\
z_1
\end{pmatrix}
$$
or 
$$
\begin{pmatrix}
\lambda_{00} & \lambda_{01} \\
\lambda_{10} & \lambda_{11}
\end{pmatrix}
\begin{pmatrix}
y_0 \\
y_1
\end{pmatrix}
\times
\begin{pmatrix}
\lambda_{00} & \lambda_{01} \\
\lambda_{10} & \lambda_{11}
\end{pmatrix}
\begin{pmatrix}
x_0 \\
x_1
\end{pmatrix}
\times
\begin{pmatrix}
\mu_{00} & \mu_{01} \\
0 & \mu_{11}
\end{pmatrix}
\begin{pmatrix}
z_0 \\
z_1
\end{pmatrix}
\ ,
$$
for some matrices $(\lambda_{ij})_{0\leq i,j \leq 1}, \, (\mu_{ij})_{0\leq i,j \leq 1}$ in $\PGL(2,\C)$,
depending on whether the tri-degrees in $X$ are preserved or permuted. In particular, it follows that $\text{Stab}(C_o)$ acts transitively on each of the orbits in \eqref{eq: orbits in X of stabilizer of (1,1,0)}. On the other hand, 
the last three orbits are respectively transformed by $\zeta_W$ into $P$, $H_{1-2}\cap H_4$, and $Y_W\backslash H_{1-2}$. Thus, we find three of the orbits in the statement. The orbit $H_{1-2}\backslash (H_4\cup P)$ corresponds to the projection by $\Pi_W$ of general orbit in $E_W$, when we let $\text{Stab}(C_o)$ act on $\blowup_W X$.
\end{proof}
\subsubsection{Orbits in $Y_{C_\times}$} In this subsection, we set $W = W_\times$. The birational map $\zeta_W$ in Diagram \eqref{eq: commutative diagram} is
$$
\zeta_{W}: (x_0:x_1)\times (y_0:y_1)\times (z_0:z_1) \mapsto ( x_0y_0z_1 : x_1y_0z_1 : x_0y_1z_1 : x_1y_1z_1 : x_1y_1z_0 )\ .
$$
Once more, $Y_W$ is determined by $t_0t_3 - t_1t_2 = 0$. However, $\Pi_W (E_W)$ is now the union of the loci $t_1 = t_3 = 0$ and $t_2 = t_3 = 0$ in $\P_\C^4$. We maintain the notation of §\ref{subsection: orbits nonsingular (1,1,0)}, and write $H_i$ for the divisor in $Y_W$ defined by $t_i = 0$. Additionally, we set $O=$ ($1\,$:$\,0\,$:$\,0\,$:$\,0\,$:$\,0$) and write $r$ for the projective line $\overline{OP} = H_1 \cap H_2 \cap H_3$.

\begin{lemma}
\label{lemma: (1,1,2) orbits in Y_W reducible}
The group action of $\text{Stab}(C_\times)$ on $Y_W$ determines the six orbits: 
\begin{gather*}
Y_W \backslash ( (H_1 \cup H_2)\cap H_3 )\, ,\, ( (H_1 \cup H_2)\cap H_3 ) \backslash (H_4  \cup  r)\ ,\\
((H_1 \cup H_2)\cap H_3 \cap H_4)\backslash O \ ,\ r\, \backslash (O  \cup  P )\, ,\, O\, ,\, P
\end{gather*}
\end{lemma}
\begin{proof}
Write $D_x$, $D_y$, and $D_z$ for the divisors in $X$ respectively determined by $x_1 = 0$, $y_1 = 0$, and $z_1 = 0$, and moreover let $Q' = (1:0)^3\in X$. The group action of $\text{Stab}(C_\times)$ on $X$ determines the five orbits 
$$
Q'\ ,\ C_\times \backslash Q'\ ,\ D_{z} \backslash C_\times \ ,\ (D_x \cup D_y)\backslash C_\times \ ,\ X \backslash (D_x\cup D_y\cup D_z)\ .
$$
The last three orbits are respectively transformed by $\zeta_W$ into $P$, $((H_1 \cup H_2)\cap H_3\cap H_4)\backslash O$, and $Y_W\backslash ( (H_1 \cup H_2)\cap H_3 )$. Therefore, we find three of the orbits in the statement. The orbits $( (H_1 \cup H_2)\cap H_3 ) \backslash (H_4\cup r)$, $r \backslash (O \cup P)$, and $O$ correspond to the projection by $\Pi_W$ of the orbits in $E_W$, when we let $\text{Stab}(C_\times)$ act on $\blowup_W X$.
\end{proof}

The following is a classification theorem for the tri-linear birational maps of type $(1,1,2)$, $(1,2,1)$, and $(2,1,1)$. Equivalently, it provides the orbits of the right-action in $\Birr_{(1,1,2)}$, $\Birr_{(1,2,1)}$, and $\Birr_{(2,1,1)}$.

\begin{theorem}
\label{theorem: classification (1,1,2)}
Let $\phi$ be birational of type either $(1,1,2)$, $(1,2,1)$ or $(2,1,1)$. Then, $\phi$ belongs to the orbit of one of the following birational maps:
\begin{itemize}
    \item $\rho_1^{(1,1,2)} \equiv ( x_1 y_1 z_1 : x_0 y_1 z_1 : x_0 y_0 z_1 : x_1 y_0 z_0 - x_0 y_1 z_0 )$ 
    \vskip0.0cm
    \noindent A surface in $W_{\rho_1}$ contains an irreducible curve of tri-degree $(1,1,0)$ and an isolated point. The base ideal is
    $$
    (x_0 y_1 - x_1 y_0, z_1)\cap (x_0, y_1, z_0)\ .
    $$
    \item $\rho_2^{(1,1,2)} \equiv (  x_0 y_1 z_1 : x_1 y_0 z_1 : x_0 y_0 z_1 : x_1 y_0 z_0 -x_0 y_1 z_0 - x_1 y_1 z_1 )$
    \vskip0.0cm
    \noindent A surface in $W_{\rho_2}$ contains an irreducible curve of tri-degree $(1,1,0)$ and has contact to a surface of tri-degree $(1,1,1)$ at a point of the curve. The base ideal is
    $$
    (x_0 y_1 - x_1 y_0, z_1)\, \cap  (z_1^2 , y_0 z_1, x_0 z_1, y_0^2 , x_0 y_0, x_0^2 , x_1 y_0 z_0 - x_0 y_1 z_0 - x_1 y_1 z_1)\ .
    $$
    \item $\rho_3^{(1,1,2)} \equiv ( x_0 y_1 z_1 : x_1 y_0 z_1 : x_0 y_0 z_1 : x_1 y_1 z_0)$
    \vskip0.0cm
    \noindent A surface in $W_{\rho_3}$ contains a pair of intersecting lines, of tri-degrees $(1,0,0)$ and $(0,1,0)$, and an isolated point. The base ideal is
    $$
    (x_1,z_1)\cap (y_1,z_1)\cap (x_0,y_0,z_0)\ .
    $$
    \item $\rho_4^{(1,1,2)} \equiv (x_1 y_1 z_1 : x_0 y_1 z_1: x_0 y_0 z_1:  x_1 y_1 z_0 - x_1 y_0 z_1 )$
    \vskip0.0cm
    \noindent A surface in $W_{\rho_4}$ contains a pair of intersecting lines, of tri-degrees $(1,0,0)$ and $(0,1,0)$, and is tangent to a curve of tri-degree $(0,1,1)$ at one of the points of the line of tri-degree $(1,0,0)$. The base ideal is
    $$
     (x_1,z_1)\cap (y_1,z_1)\cap (x_0, y_0 z_1 - y_1 z_0, y_1^2 , y_1 z_1,  y_1^2)\ .
    $$
    \item $\rho_5^{(1,1,2)} \equiv ( x_1 y_1 z_1 : x_0 y_1 z_1 : x_1 y_0 z_1 : x_1 y_1 z_0 + x_0 y_0 z_1 )$
    \vskip0.0cm
    \noindent A surface in $W_{\rho_5}$ contains a pair of intersecting lines, of tri-degrees $(1,0,0)$ and $(0,1,0)$, and has contact to a surface of tri-degree $(1,1,1)$ at the point of intersection of the lines. The base ideal is
    $$
    (x_1,z_1)\cap (y_1,z_1)\,\cap 
    (x_1 y_1 z_0 + x_0 y_0 z_1 , x_1^2 , x_1 z_1, y_1^2  , y_1 z_1 , z_1^2)\ .
    $$
\end{itemize}
\end{theorem}

\begin{proof}
From Diagram \eqref{eq: commutative diagram}, $\phi$ lies in the orbit of a birational map that factors as $\zeta_W \circ \pi_L$ for either $W = W_o$ or $W = W_\times$ and some point $L\in Y_W$. In particular, the orbit of $\phi$ is determined by one of the orbits listed in Lemmas \ref{lemma: (1,1,2) orbits in Y_W irreducible} and \ref{lemma: (1,1,2) orbits in Y_W reducible}. In the case that $L = P$, we find a common factor to the entries of $\phi$.  Excluding this case, the possibilities when $W = W_o$ are the following:
\begin{itemize}
    \item If $L \in Y_W \backslash H_{1-2}$, $\phi$ lies in the orbit of $\rho_1^{(1,1,2)}$
    \item If $L \in H_{1-2} \backslash (H_4 \cup P)$, $\phi$ lies in the orbit of $\rho_2^{(1,1,2)}$
    \item If $L \in H_{1-2}\cap H_4$, $\phi$ lies in the orbit of $\rho_2^{(1,1,1)}$
\end{itemize}
Secondly, if $W = W_\times$ the possibilities are the following:
\begin{itemize}
    \item If $L \in Y_W \backslash ( (H_1\cup H_2)\cap H_3 )$, $\phi$ lies in the orbit of $\rho_3^{(1,1,2)}$
    \item If $L \in ( (H_1\cup H_2)\cap H_3 ) \backslash (H_4 \cup r)$, $\phi$ lies in the orbit of $\rho_4^{(1,1,2)}$
    \item If $L\in ((H_1 \cup H_2) \cap H_3 \cap H_4) \backslash O$, $\phi$ lies in the orbit of $\rho_3^{(1,1,1)}$
    \item If $L \in r \backslash (O \cup P)$, $\phi$ lies in the orbit of $\rho_5^{(1,1,2)}$
    \item If $L = O$, $\phi$ lies in the orbit of $\rho_4^{(1,1,1)}$
\end{itemize}
\end{proof}

\subsection{Orbits of birational maps of type $(1,2,2)$, $(2,1,2)$ and $(2,2,1)$}
\label{subsection: (1,2,2) orbits}

By Lemma \ref{lemma: canonical CM curves}, any Cohen-Macaulay curve in $X$ of tri-degree either $(1,0,0)$, $(0,1,0)$, or $(0,0,1)$ is a projective line, and by means of an automorphism of $X$ it can be transformed into the line $\ell$ of ideal $I_\ell = (y_1,z_1)$. In particular, any birational map of type $(1,2,2)$, $(2,1,2)$, or $(2,2,1)$ lies in the orbit of a birational map with $\ell$ in its base locus.
\vskip0.2cm
According to Diagram \eqref{eq: commutative diagram}, in order to study the orbits of the right-action in $\Birr_{(1,2,2)}$, $\Birr_{(2,1,2)}$, and $\Birr_{(2,2,1)}$ we can  compute the orbits of the group action given by the stabilizer $\text{Stab}(\ell) \leq \Aut(X)$ of $\ell$ acting on $Y_W$, where $W = {(I_\ell)}_{(1,1,1)}$. 
Namely, the birational map in the diagram $\zeta_W : X \dashrightarrow Y_W \subset \P_\C^5$ is 
$$
\zeta_W :
(x_0:x_1)\times (y_0:y_1)\times (z_0:z_1) \mapsto (
x_0 y_0 z_1 :
x_1 y_0 z_1 :
x_0 y_1 z_1 : 
x_1 y_1 z_1 :
x_0 y_1 z_0 :
x_1 y_1 z_0
)\ ,
$$
and $Y_W$ is determined by the ideal
$
(t_0t_5 - t_1t_4 , t_2t_5 - t_3t_4 , t_1t_2 - t_0t_3)
$. 
The defining equations of $\Pi_W(E_W)$ in $Y_W$ are $t_2 = t_3 = 0$. We write $H_{01}$, $H_{23}$, and $H_{45}$ for the subvarieties of $Y_W$ given respectively by $t_0 = t_1 = 0$, $t_2 = t_3 = 0$, and $t_4 = t_5 = 0$.

\begin{lemma}
\label{lemma: (1,2,2) orbits on Y_W}
The group action of $\text{Stab}(\ell)$ on $Y_W$ determines the three orbits:
$$
Y_W \backslash H_{23}\ ,\ H_{23}\backslash (H_{01}\cup H_{45})\ ,\ H_{23}\cap (H_{01}\cup H_{45})
$$
\end{lemma}

\begin{proof}
Write $D_y$ and $D_z$ for the divisors in $X$ respectively determined by $y_1 = 0$ and $z_1 = 0$. The group action of $\text{Stab}(\ell)$ on $X$ determines the three orbits 
$$
\ell\ ,\ (D_y \cup D_z) \backslash \ell\ ,\ X \backslash (D_y \cup D_z)\ .
$$
The last two orbits are respectively transformed by $\zeta_W$ into $H_{23}\cap (H_{01}\cup H_{45})$ and $Y_W \backslash H_{23}$. Thus, we find two of the orbits in the statement. The orbit $H_{23}\backslash (H_{01}\cup H_{45})$ corresponds to the projection by $\Pi_W$ of the general orbit in $E_W$, when we let $\text{Stab}(\ell)$ act on $\blowup_W X$.
\end{proof}

The following is a classification theorem for the tri-linear birational maps of type $(1,2,2)$, $(2,1,2)$, and $(2,2,1)$. Equivalently, it provides the orbits of the right-action in $\Birr_{(1,2,2)}$, $\Birr_{(2,1,2)}$, and $\Birr_{(2,2,1)}$.

\begin{theorem}
\label{theorem: classification (1,2,2)}
Let $\phi$ be birational of type either $(1,2,2)$, $(2,1,2)$ or $(2,2,1)$. Then, $\phi$ belongs to the orbit of one of the following birational maps:
\begin{itemize}
    \item $\rho_1^{(1,2,2)} \equiv ( x_1 y_0 z_1 - x_0 y_1 z_1 : x_0 y_0 z_1 - x_0 y_1 z_1 : x_1 y_1 z_0 - x_0 y_1 z_1 : x_0 y_1 z_0 - x_0 y_1 z_1 )$ 
    \vskip0.1cm
    \noindent A surface in $W_{\rho_1}$ contains a line and two isolated points with all the coordinates different, i.e$.$ the projections onto the factors of $X$ are all different. The base ideal is
    $$
    (y_1, z_1)\cap (x_0, y_0, z_0) \cap (x_0 - x_1, y_0 - y_1, z_0 - z_1)\ .
    $$
    \item $\rho_2^{(1,2,2)} \equiv ( x_1 y_0 z_1: x_0 y_0 z_1 : x_1 y_1 z_0 - x_0 y_1 z_1 : x_0 y_1 z_0 + x_0 y_1 z_1 )$ 
    \vskip0.1cm
    \noindent A surface in $W_{\rho_2}$ contains a line and two isolated points with the same coordinate in exactly one of the factors of $X$. The base ideal is
    $$
    (y_1, z_1)\cap (x_0, y_0, z_0) \cap (x_0 + x_1, y_0, z_0 + z_1)\ .
    $$
    \item $\rho_3^{(1,2,2)} \equiv ( x_1 y_0 z_1 - x_0 y_1 z_1 : x_0 y_0 z_1 : x_1 y_1 z_0 - x_0 y_1 z_1 : x_0 y_1 z_0 )$ 
    \vskip0.1cm
    \noindent A surface in $W_{\rho_3}$ contains a line, and is tangent to a curve of tri-degree $(1,1,1)$ at an isolated point. The base ideal is
    $$
    (y_1, z_1) \cap (x_0 y_1 - x_1 y_0 , 
    x_0 z_1 - x_1 z_0 ,
    y_0 z_1 - y_1 z_0 ,
    x_0^2 , x_0 y_0 , x_0 z_0 , y_0^2 , y_0 z_0 , z_0 ^2)\ .
    $$
    \item $\rho_4^{(1,2,2)} \equiv ( x_1 y_0 z_1 - x_0 y_1 z_1 : x_0 y_0 z_1 : x_1 y_1 z_0 : x_0 y_1 z_0 )$ 
    \vskip0.1cm
    \noindent A surface in $W_{\rho_4}$ contains a line of tri-degree $(1,0,0)$ and is tangent to a curve of tri-degree $(1,1,0)$ at an isolated point. The base ideal is
    $$
    (y_1, z_1)
    \cap
    (z_0,x_0 y_1 - x_1 y_0 , x_0^2 , x_0 y_0 , y_0^2)\ .
    $$
    \item $\rho_5^{(1,2,2)} \equiv ( x_0 y_1 z_1 : x_1 y_0 z_1 : x_1 y_1 z_0 : x_0 y_1 z_0 - x_0 y_0 z_1 )$ 
    \vskip0.0cm
    \noindent A surface in $W_{\rho_5}$ contains a line of tri-degree $(1,0,0)$ and an isolated point, and is tangent to a curve of tri-degree $(0,1,1)$ at a point of the line. The base ideal is
    $$
    (y_1, z_1)\cap (x_0,y_0,z_0) \cap (x_1, y_0 z_1 - y_1 z_0, y_1^2, y_1 z_1 ,z_1^2 )\ .
    $$
    \item $\rho_6^{(1,2,2)} \equiv ( x_1 y_1 z_1 : x_0 y_1 z_1 : x_1 y_1 z_0 + x_1 y_0 z_1 : x_0 y_1 z_0 - x_0 y_0 z_1 - 2 x_1 y_0 z_1 )$ 
    \vskip0.0cm
    \noindent A surface in $W_{\rho_6}$ contains a line of tri-degree $(1,0,0)$ and is tangent to two curves of tri-degree $(0,1,1)$ at two distinct points of the line. The base ideal is
    $$
    (y_1, z_1) 
    \cap 
    (x_1, y_0 z_1 - y_1 z_0, y_1^2, y_1 z_1 ,z_1^2 )
     \cap (x_0 + x_1, y_0 z_1 + y_1 z_0, y_1^2, y_1 z_1 ,z_1^2 )
    \ .
    $$
    \item $\rho_7^{(1,2,2)} \equiv ( x_1 y_1 z_1 : x_0 y_1 z_1 : x_1 y_1 z_0 - x_1 y_0 z_1 : x_0 y_1 z_0 - x_0 y_0 z_1 + x_1 y_0 z_1 )$
    \vskip0.0cm
    \noindent A surface in $W_{\rho_7}$ contains a line and has contact to a surface of tri-degree $(0,1,1)$ at a point of the line. The base ideal is
    $$
    (y_1, z_1) \cap  ( x_1 y_1 z_0 - x_1 y_0 z_1 , x_0 y_1 z_0 - x_0 y_0 z_1 + x_1 y_0 z_1 , x_1^2 , y_1^2 , y_1 z_1 , z_1^2 )\ .
    $$
    \item $\rho_8^{(1,2,2)} \equiv ( x_1 y_1 z_1 : x_1 y_0 z_1 - x_0 y_1 z_1 : x_1 y_1 z_0 - x_0 y_1 z_1 : x_0 y_1 z_0 - x_0 y_0 z_1 - x_0 y_1 z_1 )$ 
    \vskip0.0cm
    \noindent A surface in $W_{\rho_8}$ contains a line and has contact to a surface of tri-degree $(0,1,1)$ at a point of the line. The base ideal is
    \begin{gather*}
    (y_1, z_1)\, \cap
    (y_1z_0 - y_0z_1 - y_1z_1, x_1^2 , z_1^3 , y_1z_1^2 , x_1z_1^2 , x_1z_0z_1 - x_0z_1^2 , \\ y_1^2 z_1, x_1y_1z_1, x_1y_0z_1 - x_0y_1z_1, y_1^3 , x_1y_1^2 , x_1y_0y_1 - x_0y_1^2 )
    \ .
    \end{gather*}
\end{itemize}
\end{theorem}

\begin{proof}
From Diagram \eqref{eq: commutative diagram}, $\phi$ lies in the orbit of a birational map that factors as $\zeta_W\circ \pi_L$ for some line $L\subset \P_\C^5$ such that the restriction $\pi_L|_{Y_W}$ has degree one, i.e$.$ $\pi_L|_{Y_W}$ is birational. From Proposition \ref{proposition: commutative diagram}, we have $\deg Y_{ W } = 3$. 
If $L\not\subset Y_W$, then $\deg (L\cap Y_W) \geq 2$ as otherwise $\deg (\pi_L|_{Y_W}) > 1$. On the other hand, consider a linear parametrization $\P_\C^1\xrightarrow{} L$. As the ideal of $Y_W$ is generated by three quadratic forms, the pull-back of $L\cap Y_W$ yields three  quadratic forms in $\P_\C^1$, implying that $\deg(L\cap Y_W)\leq 2$. 
Then, we must have either $L\subset Y_W$ or $L\not\subset Y_W$ and $\deg(L\cap Y_W) = 2$.
\vskip0.2cm
Given a point $P\in Y_W\backslash H_{23}$, we find $Q = (\alpha_0:\alpha_1)\times (\beta:1)\times(\gamma:1) \in X$ such that $\zeta_W(Q) = P$. Write $Y_y(P) \subset Y_W$ for the closure of the image of the restriction of $\zeta_W$ to $\P_\C^1\times (\beta : 1) \times \P_\C^1$. Analogously, define $Y_z(P) \subset Y_W$ by restricting $\zeta_W$ to $\P_\C^1\times \P_\C^1 \times (\gamma : 1)$. Notice that the restriction of $\zeta_W$ to $(\alpha_0:\alpha_1)\times \P_\C^1 \times \P_\C^1$ is a projective plane. Both $Y_y(P)$ and $Y_z(P)$ are smooth subvarieties of $Y_W$. Given a subvariety $V\subset Y_W$ and a smooth point $P\in V$, we write $T_P (V)$ for the tangent space of $V$ at $P$. Now, we discuss all the possible intersections between $L$ and $Y_W$. In the first place, we assume that $L\not\subset Y_W$:
\begin{itemize}
    \item If $L$ meets $Y_W$ transversally at $P_1,P_2 \in Y_W \backslash H_{23}$, and moreover we have   $P_2\not\in Y_y(P_1)$ and $P_2\not\in Y_z(P_1)$, $\phi$ lies in the orbit of $\rho_1^{(1,2,2)}$
    
    \item If $L$ meets $Y_W$ transversally at $P_1,P_2 \in Y_W \backslash H_{23}$, and either $P_2\in Y_y(P_1)$ or $P_2 \in Y_z(P_1)$, $\phi$ lies in the orbit of $\rho_2^{(1,2,2)}$
    
    \item If $L\subset T_P(Y_W)$ at a point $P \in Y_W \backslash H_{23}$, and moreover we have $L\not\subset T_P (Y_y(P))$ and $L\not\subset T_P (Y_z(P))$,  
    $\phi$ lies in the orbit of $\rho_3^{(1,2,2)}$
    
    \item If $L\subset T_P(Y_W)$ at a point $P \in Y_W \backslash H_{23}$, and either $L\subset T_P (Y_y(P))$ or $L\subset T_P (Y_z(P))$, 
    
    $\phi$ lies in the orbit of $\rho_4^{(1,2,2)}$
    \item If $L$ meets $Y_W$ transversally at $P_1 \in Y_W \backslash H_{23}$, $P_2 \in H_{23} \backslash (H_{01}\cup H_{45})$, $\phi$ lies in the orbit of $\rho_5^{(1,2,2)}$
    
    \item If $L$ meets $Y_W$ transversally at $P_1,P_2 \in H_{23} \backslash (H_{01}\cup H_{45})$, $\phi$ lies in the orbit of $\rho_6^{(1,2,2)}$
    
    \item If $L \subset T_P (Y_W)$ at a point $P \in H_{23} \backslash (H_{01} \cup H_{45})$ and $L\subset T_P(H_{23})$, $\phi$ lies in the orbit of $\rho_7^{(1,2,2)}$
    
    \item If $L \subset T_P (Y_W)$ at a point $P \in H_{23} \backslash (H_{01} \cup H_{45})$ and $L\not\subset T_P(H_{23})$, $\phi$ lies in the orbit of $\rho_8^{(1,2,2)}$
    
    \item If $L$ meets $Y_W$ transversally at $P_1 \in Y_W \backslash H_{23}$, $P_2 \in H_{23} \cap (H_{01}\cup H_{45})$, $\phi$ lies in the orbit of $\rho_3^{(1,1,2)}$
    
    \item If $L$ meets $Y_W$ transversally at $P_1 \in H_{23} \backslash (H_{01} \cup H_{45})$, $P_2 \in H_{23} \cap (H_{01}\cup H_{45})$, $\phi$ lies in the orbit of $\rho_4^{(1,1,2)}$
    
    \item If $L\subset T_P(H_{23})$ at a point $P \in H_{23} \cap (H_{01} \cup H_{45})$, $\phi$ lies in the orbit of $\rho_5^{(1,1,2)}$
\end{itemize}
Secondly, we study the cases where $L\subset Y_W$:
\begin{itemize}
    \item If $L \subset Y_W \backslash H_{23}$, $\phi$ is not dominant
    
    \item If $L \subset Y_W$ and intersects transversally $H_{23} \backslash (H_{01} \cup H_{45})$, then $\phi$ lies in the orbit of $\rho_2^{(1,1,1)}$
    
    \item If $L \subset Y_W$ and intersects transversally $H_{23} \cap (H_{01}\cup H_{45})$, $\phi$ lies in the orbit of $\rho_3^{(1,1,1)}$
    
    \item If $L \subset H_{23} \backslash (H_{01} \cup H_{45})$, and $L\not\subset H_{23} \cap (H_{01} \cup H_{45})$, $\phi$ lies in the orbit of $\rho_4^{(1,1,1)}$
\end{itemize}
\end{proof}

\subsection{Orbits of birational maps of type $(2,2,2)$}
\label{subsection: (2,2,2) orbits}

Given a point $Q\in X$ and $\lambda,\mu,\nu\in\C^*$, by means of an automorphism of $X$ the ideal $I_{Q(\lambda,\mu,\nu)}$ can be transformed into 
$$
I'
=
I_{Q'(1,1,1)}
=
(
x_1 y_0 z_0 + x_0 y_1 z_0 + x_0 y_0 z_1
\, , \,
x_0 y_1 z_1 
\, , \,
x_1 y_0 z_1 
\, , \,
x_1 y_1 z_0 
\, , \,
x_1 y_1 z_1
)
\ ,
$$
where $Q' = (1:0)^3\in X$. In particular, any birational map of type $(2,2,2)$ lies in the orbit of a birational map with the non-reduced point defined by $I_{Q'(1,1,1)}$ in its base locus.
\vskip0.2cm
Let $\text{Stab}(Q'(1,1,1)) \leq \Aut(X)$ be the stabilizer of the non-reduced point  defined by $I'$. 
In order to study the orbits of the right-action in $\Birr_{(2,2,2)}$, we can compute the orbits of the group action given by the stabilizer $\text{Stab}(Q'(1,1,1))$ acting on $Y_W$, where we set $W = {I'}_{(1,1,1)}$. 
The birational map $\zeta_W : X \dashrightarrow Y_W \subset \P_\C^5$ in Diagram \eqref{eq: commutative diagram} is now 
$$
(x_0:x_1)\times (y_0:y_1)\times (z_0:z_1) \mapsto (
x_1 y_0 z_0 + x_0 y_1 z_0 + x_0 y_0 z_1 :
x_0 y_1 z_1 :
x_1 y_0 z_1 :
x_1 y_1 z_0 :
x_1 y_1 z_1
)\ ,
$$
and the defining equation of $Y_W$ in $\P_\C^4$ is $t_1 t_2 + t_1 t_3 + t_2 t_3 - t_0 t_4 = 0$. 
On the other hand, $\Pi_W(E_W)$ is the hyperplane section in $Y_W$ given by $t_4 = 0$, and we write $H_4 = \Pi_W(E_W)$. The divisor $H_4$ is again a two-dimensional cone, of vertex $O =$  ($1\,$:$\,0\,$:$\,0\,$:$\,0\,$:$\,0$). Similarly, we define $r_{12}$, $r_{13}$ and $r_{23}$ to be the projective lines in $Y_W$ defined by $t_1 = t_2 = t_4 = 0$, $t_1 = t_3 = t_4 = 0$, and $t_2 = t_3 = t_4 = 0$, respectively.

\begin{lemma}
\label{lemma: (2,2,2) orbits on Y_W}
The group action of $\text{Stab}(Q'(1,1,1))$ on $Y_W$ determines the four orbits: 
$$
Y_W \backslash H_{4}\ ,\ H_4 \backslash (r_{12} \cup r_{13} \cup r_{23})\ ,\ (r_{12} \cup r_{13} \cup r_{23}) \backslash O\ ,\  O
$$
\end{lemma}

\begin{proof}
Write $D_x$, $D_y$, and $D_z$ for the divisors in $X$ defined respectively by $x_1 = 0$, $y_1 = 0$, and $z_1 = 0$. Similarly, write $D_{xy}$, $D_{xz}$ and $D_{yz}$ for their intersections, with the obvious notation. The group action of $\text{Stab}(Q'(1,1,1))$ on $X$ determines the four orbits 
$$
Q'\ ,\ (D_{xy}\cup D_{xz}\cup D_{yz})\backslash Q'\ ,\ (D_{x}\cup D_{y}\cup D_{z})\backslash (D_{xy}\cup D_{xz}\cup D_{yz})\ ,\ X \backslash (D_{x}\cup D_{y}\cup D_{z})\ .
$$
The last three orbits are respectively transformed by $\zeta_W$ into $O$, $(r_{12} \cup r_{13} \cup r_{23}) \backslash O$, and $Y_W \backslash H_{4}$. Therefore, we find three of the orbits in the statement. The orbit $H_4 \backslash (r_{12} \cup r_{13} \cup r_{23})$  corresponds to the projection by $\Pi_W$ of the general orbit in $E_W$, when we let $\text{Stab}(Q(1,1,1))$ act on $\blowup_W X$.
\end{proof}

The following is a classification theorem for the tri-linear birational maps of type $(2,2,2)$. Equivalently, it provides the orbits of the right-action in $\Birr_{(2,2,2)}$.

\begin{theorem}
\label{theorem: classification (2,2,2)}
Let $\phi$ be birational of type $(2,2,2)$. Then, $\phi$ belongs to the orbit of one of the following birational maps:
\begin{itemize}
    \item $\rho_1^{(2,2,2)} \equiv ( x_0 y_1 z_1 : x_1 y_0 z_1 : x_1 y_1 z_0 : x_1 y_0 z_0 + x_0 y_1 z_0 + x_0 y_0 z_1 )$ 
    \vskip0.0cm
    \noindent A surface in $W_{\rho_1}$ contains a point and has contact to a surface of tri-degree $(1,1,1)$ at a distinct point. The base ideal is
    $$
    (x_0, y_0, z_0)\cap (x_1 y_0 z_0 + x_0 y_1 z_0 + x_0 y_0 z_1 \, ,\, x_1^2\, ,\, x_1 y_1\, ,\, x_1 z_1\, ,\, y_1^2\, ,\, y_1 z_1\, ,\, z_1^2 )\ .
    $$
    \item $\rho_2^{(2,2,2)} \equiv 
        ( x_1 y_1 z_1 : x_1 y_0 z_1 - x_0 y_1 z_1 : 2x_1 y_1 z_0 + x_0 y_1 z_1 : x_1 y_0 z_0 + x_0 y_1 z_0 + x_0 y_0 z_1)$ 
    \vskip0.0cm
    \noindent A surface in $W_{\rho_2}$ has contact to a surface of tri-degree $(1,1,1)$ at a point. The base ideal is
    \begin{gather*}
    (
    z_1^3\, ,\, 
y_1 z_1^2\, ,\, 
x_1 z_1^2\, ,\, 
2 y_1 z_0 z_1 + y_0 z_1^2\, ,\, 
2 x_1 z_0 z_1 + x_0 z_1^2\, ,\,
y_1^2 z_1\, ,\, 
x_1 y_1 z_1\, ,\, 
x_1 y_0 z_1 - x_0 y_1 z_1\, ,\, \\
x_1^2 z_1\, ,\, 
2 y_1^2 z_0 + y_0 y_1 z_1\, ,\, 
2 x_1y_1 z_0 + x_0 y_1 z_1\, ,\,
x_1 y_0 z_0 + x_0 y_1 z_0 + x_0 y_0 z_1\, ,\, \\
2 x_1^2 z_0 + x_0 x_1 z_1\, ,\, 
y_1^3\, ,\, 
x_1 y_1^2 \, , \, 
x_1 y_0 y_1 -x_0 y_1^2\, ,\, 
x_1^2 y_1\, ,\, 
x_1^2 y_0 - x_0 x_1 y_1 \, ,\, 
x_1^3
    )\ .
\end{gather*}
\end{itemize}
\end{theorem}

\begin{proof}
From Diagram \eqref{eq: commutative diagram}, $\phi$ lies in the orbit of a birational map that factors as $\zeta_W\circ \pi_L$ for some point $L\in Y_W$. We have the following four possibilities:
\begin{itemize}
    \item If $L\in Y_W\backslash H_4$, $\phi$ lies in the orbit of $\rho_1^{(2,2,2)}$
    \item If $L\in H_4\backslash (r_{12} \cup r_{13} \cup r_{23})$, $\phi$ lies in the orbit of $\rho_2^{(2,2,2)}$
    \item If $L\in (r_{12} \cup r_{13} \cup r_{23}) \backslash O$, $\phi$ lies in the orbit of $\rho_2^{(1,1,2)}$
    \item If $L = O$, $\phi$ lies in the orbit of $\rho_4^{(1,1,1)}$
\end{itemize}
\end{proof}

Altogether, Theorems \ref{theorem: classification (1,1,1)}, \ref{theorem:  classification (1,1,2)}, \ref{theorem:  classification (1,2,2)}, and \ref{theorem:  classification (2,2,2)} provide the complete list of the orbits of the right-action in $\Bir_{(1,1,1)}$. Interestingly, we find the following corollary from this classification. 

\begin{corollary}
\label{corollary: unirationality}
All the irreducible components of $\Bir_{(1,1,1)}$ are unirational.
\end{corollary}
\begin{proof}
Given an irreducible component $V$ of $\Bir_{(1,1,1)}$, the action of the subgroup $H \trianglelefteq \Aut(X)$ consisting of the automorphisms that preserve the tri-degrees in $X$ on any representative of the general orbit of $V$, as listed in either Theorem  \ref{theorem: classification (1,1,1)}, \ref{theorem:  classification (1,1,2)}, \ref{theorem:  classification (1,2,2)}, or \ref{theorem:  classification (2,2,2)}, yields a dominant rational map from $H$ to $V$. As we find a dense affine set in $H\cong \PGL(2,\C)^3$, the result follows.
\end{proof}

\subsection{Degenerations of the base loci} Let $O_1,O_2 \subset \Bir_{(1,1,1)}$ be two orbits of the right-action. We say that the orbit $O_1$ \textit{degenerates to} $O_2$ if the Zariski closure of $O_1$ in $\Bir_{(1,1,1)}$ contains $O_2$. 
\vskip0.2cm
From the proofs of Theorems \ref{theorem: classification (1,1,1)}, \ref{theorem: classification (1,1,2)}, \ref{theorem: classification (1,2,2)}, and \ref{theorem: classification (2,2,2)} we deduce most of the degenerations of the orbits. \textsc{Diagram} \ref{diagram: degenerations} depicts geometric illustrations of all the degenerations of the possible base loci of a tri-linear birational map. The symbols $\textcolor{red}{\bullet}$ and $\textcolor{red}{-}$ respectively represent non-singular points and curves contained in any surface of the linear system, and $\textcolor{red}{\circ}$ represents a point of contact to a surface. Similarly, an arrow $\textcolor{red}{\rightarrow}$ represents a tangency at a point to a curve of tri-degree $(1,1,1)$, and a dashed arrow $\textcolor{red}{\dasharrow}$ a tangency to a curve of tri-degree $(1,1,0)$ or permutation.
\begin{figure}[h]
    \includegraphics[scale = 0.24]{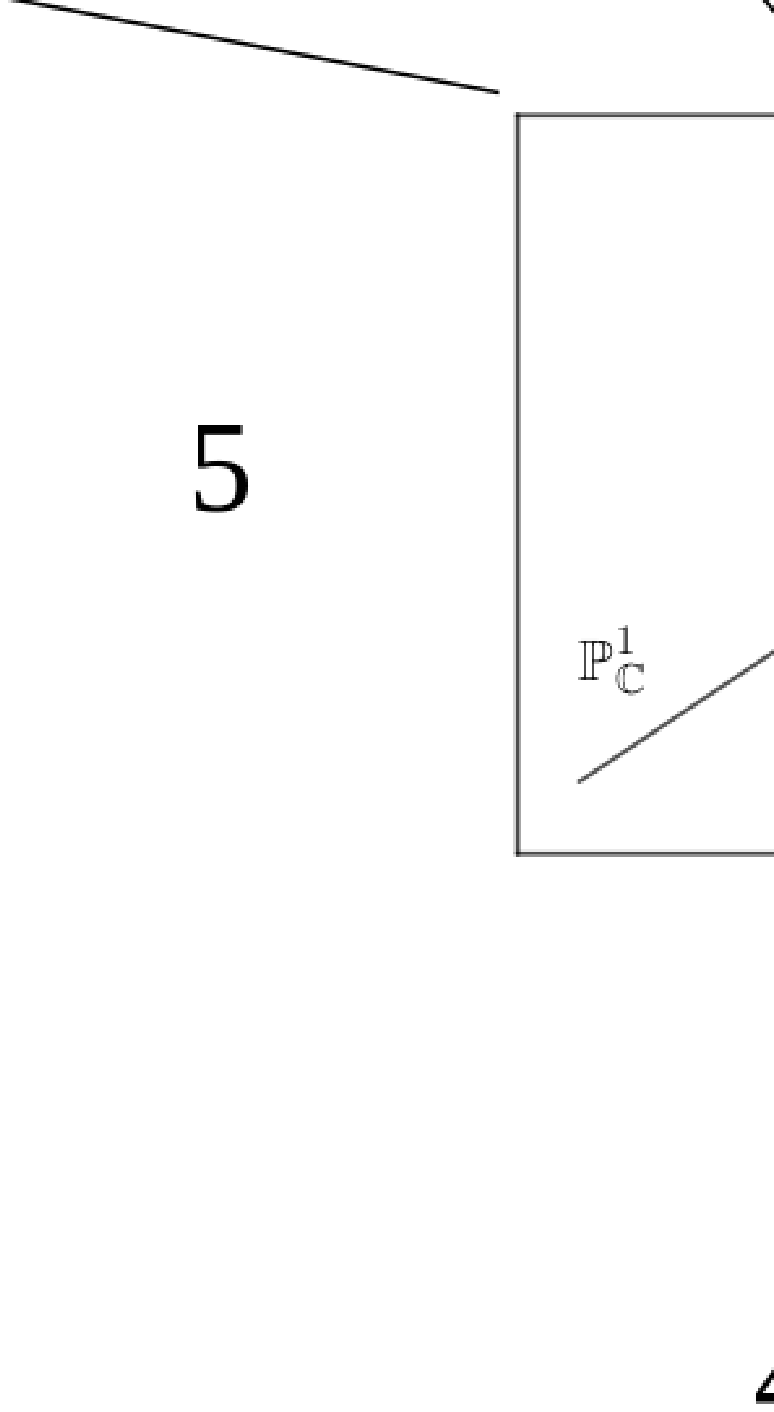}
    \captionsetup{width=\linewidth}
    \caption{Degenerations of the base loci of tri-linear birational maps. The dimension of the associated orbit of the right-action in $\Bir_{(1,1,1)}$ is written on the left. }
    \label{diagram: degenerations}
\end{figure}
\vskip0.0cm
The degenerations of the zero-dimensional base loci require further analysis. Namely, we can derive the degenerations to base loci of tri-degree $(1,0,0)$ (resp$.$ $(0,1,0)$, $(0,0,1)$) by studying the orbits in the closure $Y'$ of the image of
\begin{eqnarray*}
	\zeta': X & \dasharrow & Y' \subset \P_\C^4 \\ \nonumber
	(x_0:x_1)\times(y_0:y_1)\times(z_0:z_1) & \mapsto & 
	( x_1 y_1 z_1:
x_0 y_1 z_1 :
x_1 y_0 z_1 :
x_1 y_1 z_0 :
x_0 y_1 z_0 + x_0 y_0 z_1)
\ .
\end{eqnarray*}
In this case, the entries of $\zeta'$ determine a basis for the graded component of the ideal
$$
I' = (x_0 y_1 z_0 + x_0 y_0 z_1) + (x_1,y_1,z_1)^2
$$
in tri-degree $(1,1,1)$. Similarly, the degenerations to base loci of tri-degree $(1,1,0)$ (resp$.$ $(1,0,1)$, $(0,1,1)$) can be deduced from the orbits in the closure $Y''$ of the image of
\begin{eqnarray*}
	\zeta'': X & \dasharrow & Y'' \subset \P_\C^4 \\ \nonumber
	(x_0:x_1)\times(y_0:y_1)\times(z_0:z_1) & \mapsto & 
	(x_1 y_1 z_1 :
x_0 y_1 z_1 :
x_1 y_0 z_1 :
x_1 y_1 z_0 :
x_0 y_0 z_1)\ ,
\end{eqnarray*}
where the entries of $\zeta''$ now determine a basis of the graded component of the ideal 
$$
I'' = (x_0y_0z_1) + (x_1,y_1,z_1)^2
$$
in tri-degree $(1,1,1)$. 
As an example, the orbits on $Y'$ of the action given by the stabilizer of the subscheme defined by $I'$ are
\begin{equation}
\label{eq: orbits in degenerations}
Y' \backslash H_4 \ ,\ (H_1\cap H_4) \backslash H_{2-3} \ ,\ (H_{2-3}\cap H_4) \backslash (H_1\cup H_2) \ ,\ r \backslash H_1 \ ,\ H_1 \cap H_{2-3} \cap H_4\ ,\ O\ ,
\end{equation}
where $H_1,H_{2-3}, H_4$ are the divisors in $Y'$ respectively given by $t_1 = 0$, $t_2 - t_3 = 0$, and $t_4 = 0$, and $r \subset Y'$ is the line determined by $t_2 = t_3 = t_4 = 0$. Representatives associated to each orbit in \eqref{eq: orbits in degenerations}, in the same order, are
$$
\rho_{5}^{(1,2,2)} \ ,\  \rho_{6}^{(1,2,2)}\ ,\
\rho_{8}^{(1,2,2)}\ ,\
\rho_{2}^{(1,1,1)}\ ,\
\rho_{7}^{(1,2,2)}\ ,\
\rho_{4}^{(1,1,1)}.
$$
As $r \subset H_{2-3}\cap H_4$ but $r \not \subset H_1\cap H_4$, it follows that $\rho_{8}^{(1,2,2)}$ degenerates to $\rho_{2}^{(1,1,1)}$, but $\rho_{6}^{(1,2,2)}$ does not.

\section{Syzygies of tri-linear birational maps} 
\label{section: syzygy-based characterization}

Finally, we study the syzygies of the entries of tri-linear birational maps. Interestingly, the following theorem yields a birationality criterion for tri-linear rational maps, based on the computation of the first syzygies in some tri-degrees. Furthermore, novel methods for the manipulation of birational volumes might be devised from this theorem, relying on the imposition of specific syzygies. In order to simplify the statement, we set
$$
e_1 = (1,0,0)
\ ,\ 
e_2 = (0,1,0)
\ ,\ 
e_3 = (0,0,1)
\ .
$$
\begin{theorem}
\label{theorem: main computational result}
Assume that $\phi$ is dominant, and let $i,j,k$ be indices such that $\{ i,j,k \} = \{ 1,2,3 \}$. Then, $\phi$ is birational if and only if one of the following conditions holds:
\begin{enumerate}
\setlength\itemsep{0.5em}
\item The $f_i$'s have syzygies of tri-degrees $e_1$, $e_2$, and $e_3$. In this case, the type of $\phi$ is $(1,1,1)$.
\item The $f_i$'s have syzygies of tri-degree $e_i$ and $e_j$, but not $e_k$. In this case, the type of $\phi$ is $(1,1,1) + e_k$.
\item The $f_i$'s have a syzygy of tri-degree $e_i$, but neither $e_j$ nor $e_k$. Moreover, the $f_i$'s also have a syzygy of tri-degree either $e_i + e_j$ or $e_i + e_k$, independent from the first one. In this case, the type of $\phi$ is $(1,1,1)+e_j + e_k$.
\item The $f_i$'s do not have syzygies of tri-degree $e_1$, $e_2$, or $e_3$. Moreover, the $f_i$'s have two independent syzygies in each of the tri-degrees $e_1 + e_2$, $e_1+e_3$, and $e_2+e_3$ satisfying the splitting property in Condition \ref{condition: (2,2,2)} (see §\ref{subsection: syzygies (2,2,2)}). In this case, the type of $\phi$ is $(2,2,2)$. 
\end{enumerate}
\end{theorem}
\vskip0.2cm
The proof of Theorem \ref{theorem: main computational result} follows immediately from Propositions \ref{proposition: syzygies (1,1,1)}, \ref{proposition: syzygies (1,1,2)}, \ref{proposition: syzygies (1,2,2)}, and \ref{proposition: syzygies (2,2,2)}, and it relies strongly on the geometric classification of §\ref{section: classification}. More explicitly, the composition of a rational map with automorphisms of $(\P_\C^1)^3$ and $\P_\C^3$ preserves the structure of the module of first syzygies, with maybe a permutation in the tri-degrees (if this is not preserved by the automorphism of $(\P_\C^1)^3$). Therefore, the syzygies of the entries can be computed for the representatives of the finitely many orbits of the right-action in $\mathfrak{Bir}_{(1,1,1)}$, listed in Theorems \ref{theorem: classification (1,1,1)}, \ref{theorem: classification (1,1,2)}, \ref{theorem: classification (1,2,2)}, and \ref{theorem: classification (2,2,2)}. 
\vskip0.2cm
On the other hand, 
although Theorem \ref{theorem: main computational result} is adequate for computational purposes, we provide further cohomological information about the base loci of tri-linear birational maps. Namely, we will prove that there is a one-to-one correspondence between the type of $\phi$ and the shape of the minimal tri-graded free resolution of the base ideal  $B_\phi$. For the sake of comparison, the minimal free resolution of a tri-linear rational map $\phi$ with general entries is
$$
0\xrightarrow[]{} R \xrightarrow[]{} R^{21} \xrightarrow[]{} R^{62} \xrightarrow[]{} R^{69} \xrightarrow[]{} R ^{30}\xrightarrow[]{} R^4 \xrightarrow[]{} B_\phi \xrightarrow{} 0 \ .
$$
As we shall see, the resolution is remarkably simpler when $\phi$ is birational.

\subsection{Syzygies of birational maps of type $(1,1,1)$}

For birational maps of type $(1,1,1)$, the minimal tri-graded free resolution of the base ideal $B_\phi$ is Hilbert-Burch. 

\begin{proposition}
\label{proposition: syzygies (1,1,1)}
Assume that $\phi$ is dominant. The following are equivalent:
\vskip0.2cm
\begin{enumerate}
\setlength\itemsep{0.5em}
\item $\phi$ is birational of type $(1,1,1)$.
\item The minimal tri-graded free resolution of the base ideal $B_\phi$ is Hilbert-Burch, i.e.
\begin{equation}\label{eq:HBres}
0 \xrightarrow{} 
\begin{matrix}
R(-2,-1,-1)\\
\oplus \\
R(-1,-2,-1) \\
\oplus \\
R(-1,-1,-2)\\
\end{matrix}
\xrightarrow{} R(-1,-1,-1)^4 \xrightarrow{(f_0\ f_1\ f_2\ f_3)} B_\phi
\xrightarrow{} 0\ .
\end{equation}
In particular, the base locus $Z_\phi$ is a Cohen-Macaulay curve.
\item The $f_i$'s have syzygies of tri-degrees $e_1$, $e_2$, and $e_3$.
\end{enumerate}
\end{proposition}

\begin{proof}
We first prove that $(1)$ implies $(2)$, so we assume that $\phi$ is birational of type $(1,1,1)$. Maintaining the notation of \eqref{eq: rational inverse} for $\phi^{-1}$, in the Rees ideal $J$ we find the polynomials
\begin{equation}
\label{eq: 111 basic}
g_1 = \begin{vmatrix}
x_0 & x_1 \\
a_0 & a_1
\end{vmatrix}
\ , \
g_2 = \begin{vmatrix}
y_0 & y_1 \\
b_0 & b_1
\end{vmatrix}
\ , \
g_3 = \begin{vmatrix}
z_0 & z_1 \\
c_0 & c_1
\end{vmatrix}
\ .
\end{equation}
Consider the $4\times 3$ matrix $M = M(x_i,y_j,z_k)$, whose $(i,j)$-th entry is the coefficient of $g_j$ in $t_i$. By definition, the tuple $(f_0,f_1,f_2,f_3)$ lies in the cokernel of $M$. Now, write $\Delta_i$ for the $3\times 3$ signed minor obtained by deleting the $i$-th row of $M$. Then, the tuple $(\Delta_0,\Delta_1,\Delta_2,\Delta_3)$ also lies in the cokernel of $M$. 
Moreover, as $\phi$ is birational the specialization of the $g_i$'s at a general point of $X$ yields three independent linear forms, so the matrix $M$ has generically rank $3$. 
In particular $M$ has no kernel in $R^3$, and its cokernel is a free $R$-module of rank one generated by $(\Delta_0,\Delta_1,\Delta_2,\Delta_3)$. As the tuples $(f_0,f_1,f_2,f_3)$ and $(\Delta_0,\Delta_1,\Delta_2,\Delta_3)$ have tri-linear entries, they must differ by a non-zero constant. Therefore, as $\text{codim}\, B_\phi\geq 2$ the complex
\begin{equation}
\label{eq: complex in proof of HB}
0 \xrightarrow{} R^3 \xrightarrow{M} R^4 \xrightarrow{(\Delta_0\,\Delta_1\,\Delta_2\,\Delta_3)} B_\phi \xrightarrow{} 0
\end{equation}
is a minimal free resolution of $B_\phi$. On the other hand, from the Auslander-Buchbaum formula the ring $R/B_\phi$ is Cohen-Macaulay of codimension two (see e.g.~\cite[\S 20.4]{EISEN_COMM}), which implies that $Z_\phi$ is a Cohen-Macaulay curve. 
\vskip0.2cm
Clearly $(2)$ implies $(3)$, so we now assume $(3)$ and prove $(1)$. The syzygies of tri-degree $e_1$, $e_2$, and $e_3$ can be written as
$$
x_0
\,
a_1
-
x_1
\,
a_0
=
\begin{vmatrix}
x_0 & x_1 \\
a_0 & a_1
\end{vmatrix}
\ ,\
y_0
\,
b_1
-
y_1
\,
b_0
=
\begin{vmatrix}
y_0 & y_1 \\
b_0 & b_1
\end{vmatrix}
\ ,\
z_0
\,
c_1
-
z_1
\,
c_0
=
\begin{vmatrix}
z_0 & z_1 \\
c_0 & c_1
\end{vmatrix} 
\ ,
$$
for some linear forms $a_i,b_j,c_k$ in $\C[t_0,t_1,t_2,t_3]$. Then, the rational map from $\P_\C^3$ to $X$ given by
$$(t_0:t_1:t_2:t_3)\mapsto (a_0:a_1)\times(b_0:b_1)\times(c_0:c_1)
$$
must be $\phi^{-1}$, so $\phi$ is birational of type $(1,1,1)$.
\end{proof}

\begin{remark}
Notice that the proof of Proposition \ref{proposition: syzygies (1,1,1)} is independent from the classification of §\ref{section: classification}. In particular, it is valid to use Proposition \ref{proposition: syzygies (1,1,1)} in the proof of Theorem \ref{theorem: classification (1,1,1)}. On the other hand, it follows from the proof that the base locus $Z_\phi$ of a birational map $\phi$ of type $(1,1,1)$ is an arithmetically Cohen-Macaulay curve, i.e$.$ the tri-homogeneous coordinate ring $R/B_\phi$ is Cohen-Macaulay, which is stronger than requiring that $Z_\phi$ is Cohen-Macaulay.
\end{remark}

\subsection{Syzygies of birational maps of type $(1,1,2)$, $(1,2,1)$, and $(2,1,1)$}
\label{subsection: syzygies (1,1,2)}

The study of the syzygies of birational maps of type either $(1,1,2)$, $(1,2,1)$, or $(2,1,1)$ is equivalent, by means of an automorphism of $X$. In particular, the shape of the minimal tri-graded free resolution of the base ideal $B_\phi$ is the same, up to a permutation of the Betti numbers. Thus, we restrict to birational maps of type $(1,1,2)$.

\begin{proposition}
\label{proposition: syzygies (1,1,2)}
Assume that $\phi$ is dominant. The following are equivalent:
\vskip0.2cm
\begin{enumerate}
\setlength\itemsep{0.5em}
\item $\phi$ is birational of type $(1,1,2)$.
\item The minimal tri-graded free resolution of the base ideal $B_\phi$ is 
\begin{equation}
\label{eq: (1,1,2) FFR}
0\rightarrow R(-2,-2,-2) \xrightarrow{} 
\begin{matrix}
R(-2,-1,-1)\\
\oplus \\
R(-1,-2,-1) \\
\oplus \\
R(-2,-1,-2)\\
\oplus \\
R(-1,-2,-2) 
\end{matrix}
\xrightarrow{} R(-1,-1,-1)^4 \xrightarrow{(f_0\ f_1\ f_2\ f_3)} B_\phi
\xrightarrow{} 0\ .
\end{equation}
\item The $f_i$'s have syzygies of tri-degree $e_1$ and $e_2$, but not $e_3$.
\end{enumerate}
\end{proposition}
\begin{proof}
The implication from $(1)$ to $(2)$ follows from the computation of the minimal tri-graded free resolutions of the base ideal of all the representatives listed in Theorem \ref{theorem: classification (1,1,2)} (performed with the help of  \textsc{Macaulay2}). Additionally, it is immediate that $(2)$ implies $(3)$. 
\vskip0.2cm
Now, we assume $(3)$. The syzygies of tri-degree $e_1$ and $e_2$ can be respectively written as
\begin{equation}
\label{eq: proof (1,1,2)}
x_0
\,
a_1
-
x_1
\,
a_0
=
\begin{vmatrix}
x_0 & x_1 \\
a_0 & a_1
\end{vmatrix}
\ ,\
y_0
\,
b_1
-
y_1
\,
b_0
=
\begin{vmatrix}
y_0 & y_1 \\
b_0 & b_1
\end{vmatrix}
\ ,
\end{equation}
for some linear forms $a_i,b_j$ in $\C[t_0,t_1,t_2,t_3]$. Let $P$ be a general point in $\P_\C^3$, and $(\lambda_0:\lambda_1)$, $(\mu_0:\mu_1)\in\P_\C^1$ be the projective points such that
$$
\begin{vmatrix}
\lambda_0 & \lambda_1 \\
a_0(P) & a_1(P)
\end{vmatrix}
=
0
\ ,\ 
\begin{vmatrix}
\mu_0 & \mu_1 \\
b_0(P) & b_1(P)
\end{vmatrix}
=
0
\ .
$$
Thus, any point $Q$ in the pullback $\phi^{-1}(P)$ is of the form $Q=(\lambda_0:\lambda_1)\times (\mu_0:\mu_1)\times (z_0:z_1)$ for some $(z_0:z_1)\in\P_\C^1$, as the point $Q\times P$ in $X\times \P_\C^3$ must vanish the polynomials in \eqref{eq: proof (1,1,2)}. 
Moreover, as $\phi$ is dominant the restriction given by
\begin{eqnarray}
\label{eq: restriction of phi}
	\phi' : (\lambda_0:\lambda_1)\times (\mu_0:\mu_1) \times \P_\C^1 & \dasharrow & \P_\C^3 \\ \nonumber
 (z_0:z_1) & \mapsto & (f'_0:f'_1:f'_2:f'_3)\ ,
\end{eqnarray}
where $f'_i(z_0,z_1)= f_i(\lambda_0,\lambda_1,\mu_0,\mu_1,z_0,z_1)$, is an isomorphism of projective lines. In particular, there is a unique point in the pullback $\phi^{-1}(P)$, implying that $\phi$ is birational. 
\vskip0.2cm
The syzygies of tri-degree $e_1$ and $e_2$ determine $\phi^{-1}$ on the first two factors of $X$. Moreover there is no syzygy of tri-degree $e_3$, so the type of $\phi$ must be $(1,1,2)$.
\end{proof}

\subsection{Syzygies of birational maps of type $(1,2,2)$, $(2,1,2)$, and $(2,2,1)$}
\label{subsection: syzygies (1,2,2)}

Similarly to §\ref{subsection: syzygies (1,1,2)}, the analysis of the syzygies of birational maps of type either $(1,2,2)$, $(2,1,2)$, or $(2,2,1)$ is equivalent, and the minimal tri-graded free resolution of the base ideal $B_\phi$ is the same up to permutation of the Betti numbers. Thus, we restrict to birational maps of type $(1,2,2)$.

\begin{proposition}
\label{proposition: syzygies (1,2,2)}
Assume that $\phi$ is dominant. The following are equivalent:
\vskip0.2cm
\begin{enumerate}
\setlength\itemsep{0.5em}
\item $\phi$ is birational of type $(1,2,2)$.
\item The minimal tri-graded free resolution of the base ideal $B_\phi$ is 
\begin{equation}
\label{eq: (1,2,2) FFR}
0\rightarrow R(-2,-2,-2)^2 \xrightarrow{} 
\begin{matrix}
R(-2,-1,-1)\\
\oplus \\
R(-2,-2,-1) \\
\oplus \\
R(-2,-1,-2)\\
\oplus \\
R(-1,-2,-2)^2
\end{matrix}
\xrightarrow{} R(-1,-1,-1)^4 \xrightarrow{(f_0\ f_1\ f_2\ f_3)} B_\phi
\xrightarrow{} 0\ .
\end{equation}
\item The $f_i$'s have a syzygy of tri-degree $e_1$, but neither $e_2$ nor $e_3$. Moreover, the $f_i$'s also have a syzygy of tri-degree either $e_1 + e_2$ or $e_1 + e_3$, independent from the first one.
\end{enumerate}
\end{proposition}
\begin{proof}
The implication from $(1)$ to $(2)$ follows from the computation of the minimal tri-graded free resolutions of the base ideal of all the representatives listed in Theorem \ref{theorem: classification (1,2,2)} (performed with the help of  \textsc{Macaulay2}). Additionally, it is immediate that $(2)$ implies $(3)$. 
\vskip0.2cm
Now, we assume $(3)$. Without loss of generality, we suppose that the $f_i$'s have a syzygy of tri-degree $e_1 + e_2$. Thus, the syzygies of tri-degree $e_1$ and $e_1+e_2$ can be respectively written as
\begin{gather}
\label{eq: proof (1,2,2), 1}
x_0
\,
a_1
-
x_1
\,
a_0
=
\begin{vmatrix}
x_0 & x_1 \\
a_0 & a_1
\end{vmatrix}
\ ,
\\
\label{eq: proof (1,2,2), 2}
y_0
\,
(
x_0
\,
h_{01}
+
x_1
h_{11}
)
-
y_1
\,
(
x_0
\,
h_{00}
+
x_1
h_{10}
)
=
\begin{vmatrix}
y_0 & y_1 \\
x_0
\,
h_{00}
+
x_1
h_{10}
&
x_0
\,
h_{01}
+
x_1
h_{11}
\end{vmatrix}
\ ,
\end{gather}
for some linear forms $a_i,h_{ij}$ in $\C[t_0,t_1,t_2,t_3]$. Let $P$ be a general point in $\P_\C^3$, and $(\lambda_0:\lambda_1)\in\P_\C^1$ be the projective point such that
$$
\begin{vmatrix}
\lambda_0 & \lambda_1 \\
a_0(P) & a_1(P)
\end{vmatrix}
=
0
\ .
$$
Similarly, write $(\mu_0:\mu_1)\in\P_\C^1$ for the projective point such that
$$
\begin{vmatrix}
\mu_0 & \mu_1 \\
\lambda_0
\,
h_{00}(P)
+
\lambda_1
\,
h_{10}(P)
&
\lambda_0
\,
h_{01}(P)
+
\lambda_1
\,
h_{11}(P)
\end{vmatrix}
=
0
\ .
$$
Thus, any point $Q$ in the pullback $\phi^{-1}(P)$ is of the form $Q=(\lambda_0:\lambda_1)\times (\mu_0:\mu_1)\times (z_0:z_1)$ for some $(z_0:z_1)\in\P_\C^1$, as the point $Q\times P$ in $X\times \P_\C^3$ must vanish the polynomials  \eqref{eq: proof (1,2,2), 1} and \eqref{eq: proof (1,2,2), 2}. Repeating the argument in the proof of Proposition \ref{proposition: syzygies (1,1,2)} with the restriction map $\phi'$, it follows that there is a unique point in the pullback $\phi^{-1}(P)$, so $\phi$ is birational.
\vskip0.2cm 
The syzygy of tri-degree $e_1$ determines $\phi^{-1}$ on the first factor of $X$. On the other hand, there are no syzygies of tri-degree $e_2$ nor $e_3$, so the type of $\phi$ must be $(1,2,2)$.
\end{proof}

\subsection{Syzygies of birational maps of type $(2,2,2)$}
\label{subsection: syzygies (2,2,2)}

Finally, we study the syzygies of birational maps of type $(2,2,2)$. In this case, the entries of $\phi$ do not have syzygies of tri-degree $e_1,e_2$, or $e_3$, and the analysis is more complicated. Moreover, the tri-degree of the syzygies alone is not enough to decide birationality, as it is shown by Example \ref{example: (2,2,2)}. However, from Proposition \ref{proposition: syzygies (2,2,2)} it follows that birationality can still be decided from the first syzygies, by checking a splitting condition on three related polynomials. 
Before stating this characterization, we first establish some hypotheses on the first syzygies.

\begin{setup}
\label{setup: (2,2,2)}
Assume that the $f_i$'s do not have syzygies of tri-degree $e_1$, $e_2$, or $e_3$ but have a pair of independent syzygies in each of the tri-degrees $e_1+e_2$, $e_1+e_3$, and $e_2+e_3$. In particular, the syzygies of tri-degree $e_1+e_2$ can be written as
\begin{equation}
\label{eq: in setup (2,2,2)}
x_0
y_0
\,
g_{00}
+
x_1
y_0
\,
g_{10}
-
x_0
y_1
\,
g_{01}
-
x_1
y_1
\,
g_{11}
\ ,\ 
x_0
y_0
\,
h_{00}
+
x_1
y_0
\,
h_{10}
-
x_0
y_1
\,
h_{01}
-
x_1
y_1
\,
h_{11}
\ ,
\end{equation}
for some linear forms $g_{ij},h_{ij}$ in $\C[t_0,t_1,t_2,t_3]$. We define $u_{xy} = u_{xy}(x_i,t_j)$ as
$$
u_{xy}
\coloneqq
\begin{vmatrix}
x_0
\,
g_{00}
+
x_1
g_{10}
&
x_0
\,
g_{01}
+
x_1
g_{11}
\\
x_0
\,
h_{00}
+
x_1
h_{10}
&
x_0
\,
h_{01}
+
x_1
h_{11}
\end{vmatrix}
\ .
$$
The polynomial $u_{xy}$ 
is non-zero, as the syzygies in \eqref{eq: in setup (2,2,2)} are independent by assumption. 
\newline Moreover, we have
\begin{gather*}
\begin{pmatrix}
x_0
\,
g_{00}
+
x_1
g_{10}
&
x_0
\,
g_{01}
+
x_1
g_{11}
\\
x_0
\,
h_{00}
+
x_1
h_{10}
&
x_0
\,
h_{01}
+
x_1
h_{11}
\end{pmatrix}
\begin{pmatrix}
y_0
\\
y_1
\end{pmatrix}
\xmapsto{t_j \mapsto f_j}
\begin{pmatrix}
0
\\
0
\end{pmatrix}
\end{gather*}
implying that $u_{xy}$
is a relation in the Rees ideal $J$ of multi-degree $(2,0,0;2)$. 
Similarly, we define the relations $u_{yz},u_{zx}$ from the pair of syzygies of tri-degree $e_2+e_3$ and $e_1+e_3$, respectively. In particular, the multi-degree in $J$ of $u_{yz}$ (resp$.$ $u_{zx}$) is $(0,2,0;2)$ (resp$.$ $(0,0,2;2)$).
\end{setup}

\begin{condition}
\label{condition: (2,2,2)}
Assuming Setup \ref{setup: (2,2,2)}, at least two of the $u_{xy},u_{yz},$ and $u_{zx}$ have a linear factor in $\{x_0,x_1\}$, $\{y_0,y_1\}$, and $\{z_0,z_1\}$, respectively.
\end{condition}

\vskip0.2cm

\begin{proposition}
\label{proposition: syzygies (2,2,2)}
Assume that $\phi$ is dominant. The following are equivalent:
\vskip0.2cm
\begin{enumerate}
\setlength\itemsep{0.5em}
\item $\phi$ is birational of type $(2,2,2)$.
\item The $f_i$'s satisfy Condition \ref{condition: (2,2,2)}. In this case, the minimal tri-graded free resolution of the base ideal $B_\phi$ is 
\begin{equation}
\label{eq: (2,2,2) FFR}
0\rightarrow R(-2,-2,-2)^3 \xrightarrow{} 
\begin{matrix}
R(-2,-2,-1)^2\\
\oplus \\
R(-2,-1,-2)^2\\
\oplus \\
R(-1,-2,-2)^2
\end{matrix}
\xrightarrow{} R(-1,-1,-1)^4 \xrightarrow{(f_0\ f_1\ f_2\ f_3)} B_\phi
\xrightarrow{} 0\ .
\end{equation}
\end{enumerate}
\end{proposition}

\begin{proof}
The implication from $(1)$ to $(2)$ follows from the computation of the minimal tri-graded free resolutions of the base ideal of all the representatives listed in Theorem \ref{theorem: classification (2,2,2)}, and the polynomials $u_{xy}$, $u_{yz}$, and $u_{zx}$ always have a linear factor as in Condition \ref{condition: (2,2,2)} (these computations have been performed with the help of \textsc{Macaulay2}).
\vskip0.2cm
Now, we assume Condition \ref{condition: (2,2,2)}. 
Without loss of generality, the polynomials $u_{xy}$ and $u_{yz}$ have linear factors $l_1=l_1(x_0,x_1)$ and $l_2=l_2(y_0,y_1)$, i.e$.$ we can write
\begin{equation}
\label{eq: (2,2,2) proof, relations}
u_{xy}
=
l_1(x_0,x_1)
\,
\begin{vmatrix}
x_0 & x_1 \\
a_0 & a_1
\end{vmatrix}
\ ,\
u_{yz}
=
l_2(y_0,y_1)
\,
\begin{vmatrix}
y_0 & y_1 \\
b_0 & b_1
\end{vmatrix}
\ ,
\end{equation}
for some quadratic forms $a_i,b_j$ in $\C[t_0,t_1,t_2,t_3]$. Let $P$ given a general point in $\P_\C^3$. We now prove that $\phi^{-1}(P)$ has a unique preimage in the open set $U\subset X$ where $l_1$ and $l_2$ are both non-zero. 
Given $Q\in U\cap\phi^{-1}(P)$, the point $Q\times P\in U\times \P_\C^3$ must vanish the two polynomials in \eqref{eq: (2,2,2) proof, relations}. By assumption we have $l_1(Q)\not=0$ and $l_2(Q)\not=0$, implying that $Q\times P$ must vanish the second factor in each polynomial. In particular, $Q$ must be of the form 
$Q=(\lambda_0:\lambda_1)\times (\mu_0:\mu_1)\times (z_0:z_1)$ for some $(z_0:z_1)\in\P_\C^1$, where $(\lambda_0:\lambda_1),(\mu_0:\mu_1)\in\P_\C^1$ are the projective points satisfying 
$$
\begin{vmatrix}
\lambda_0 & \lambda_1 \\
a_0(P) & a_1(P)
\end{vmatrix}
=
0
\ ,\ 
\begin{vmatrix}
\mu_0 & \mu_1 \\
b_0(P) & b_1(P)
\end{vmatrix}
=
0
\ .
$$
Repeating the argument in the proof of Proposition \ref{proposition: syzygies (1,1,2)} with the restriction map $\phi'$, it follows that the pullback $\phi^{-1}(P)$ defines a unique point in $U$, so $\phi$ is birational. On the other hand there are no syzygies of tri-degree $e_1,e_2$ or $e_3$, so the type of $\phi$ must be $(2,2,2)$.
\end{proof}

\begin{example}
\label{example: (2,2,2)}
From Proposition \ref{proposition: syzygies (2,2,2)}, the minimal tri-graded free resolution of the base ideal of a tri-linear birational map of type $(2,2,2)$ is always \eqref{eq: (2,2,2) FFR}. However, given the rational map
\begin{gather*}
	\phi : \mathbb{P}_{\mathbb{C}}^1 \times \mathbb{P}_{\mathbb{C}}^1 \times \mathbb{P}_{\mathbb{C}}^1 \dashrightarrow \mathbb{P}_{\mathbb{C}}^3 \\
	(x_0:x_1)\times (y_0:y_1)\times (z_0:z_1) \mapsto ( x_0 y_0 z_1 : x_1 y_1 z_0 + x_1 y_0 z_1 + x_0 y_1 z_1 + x_1 y_1 z_1 : x_0 y_1 z_0 : x_1 y_0 z_0 )
\end{gather*}
the minimal tri-graded free resolution of $B_\phi$ is also as \eqref{eq: (2,2,2) FFR}, but $\phi$ is not birational as Condition \ref{condition: (2,2,2)} is not satisfied.
\end{example}

\vskip0.4cm

\printbibliography

\end{document}